\documentclass[11pt, oneside]{amsart}

\usepackage[letterpaper,body={14.5cm,22.0cm}, mag=1000]{geometry}
\usepackage[utf8]{inputenc}
\usepackage{amssymb}
\usepackage{amsthm}
\usepackage{amscd}
\usepackage{enumitem}
\usepackage{float}
\usepackage{placeins}
\usepackage{caption}
\usepackage{mathtools}

\usepackage{times}
\usepackage{graphicx}
\usepackage{tikz}
\usepackage{tikz-cd}
\usepackage[all,cmtip]{xy}

\numberwithin{equation}{section}
\theoremstyle{plain}

\newtheorem{cor}[equation]{Corollary}
\newtheorem{lemma}[equation]{Lemma}

\newtheorem{pro}[equation]{Proposition}

\newtheorem{rem}[equation]{Remark}

\newtheorem{lem}[equation]{Lemma}

\newtheorem{thm}[equation]{Theorem}

\theoremstyle{definition}

\newtheorem{example}[equation]{Example}

\newcommand{\dlabel}[1]{\ifmmode \text{\ttfamily \upshape [#1] } \else
{\ttfamily \upshape [#1] }\fi \label{#1}}

\newcommand{\C}{\operatorname{C} }

\newcommand{\Z}{\operatorname{Z} }

\newcommand{\Cb}{\operatorname{Cb} }
\newcommand{\Cy}{\operatorname{Cy} }
\newcommand{\Pb}{\operatorname{Pb} }
\newcommand{\Pt}{\operatorname{Pt} }

\newcommand{\Id}{\operatorname{Id}}
\renewcommand{\Pr}{\operatorname{Pr} }

\newcommand{\gen}[1]{\left < #1 \right >}
\newcommand{\Aut}{\operatorname{Aut} }

\newcommand{\Ker}{\operatorname{Ker} }

\newcommand{\Syl}{\operatorname{Syl} }

\newcommand{\normaleq}{\trianglelefteq}

\newcommand{\Soc}{\operatorname{Soc} }
\newcommand{\Ann}{\operatorname{Ann} }

\newcommand{\Fix}{\operatorname{Fix} }

\author{Susanta Mondal$^{1, 2,}$}
\address{$^1$ Harish-Chandra Research Institute, Chhatnag Road, Jhunsi, Prayagraj-211 019, India.}
\address{$^2$ Homi Bhabha National Institute, Training School Complex, Anushakti Nagar, Mumbai 400 094, India}
\email{susantamondal@hri.res.in}

\author{PAVEL SHUMYATSKY$^{3}$}
\address{$^3$  Department of Mathematics, University of Brasilia, DF 70910-900, Brazil}
\email{pavel@unb.br}

\author{Marco Trombetti$^{4}$}
\address{$^4$  Dipartimento di Matematica e Applicazioni “Renato Caccioppoli”, Universit`a di Napoli Federico II, Complesso Universitario Monte S. Angelo, Via
Cintia, Napoli, Italy}
\email{marco.trombetti@unina.it}

\author{Manoj K. Yadav$^{5,6}$}
\address{$^5$  Harish-Chandra Research Institute, Chhatnag Road, Jhunsi, Prayagraj-211 019, India.}
\address{$^6$ Homi Bhabha National Institute, Training School Complex, Anushakti Nagar, Mumbai 400 094, India}
\email{myadav@hri.res.in}


\begin{document}
\setcounter{page}{1}
\title[Relative commuting probability and BFC-type results for finite skew  left braces]
{Relative commuting probability and BFC-type results for finite skew braces}

\begin{abstract}

We develop a probabilistic approach to the structure of finite skew left braces, motivated both by classical commuting probability in finite group theory and by the role of skew left braces in the study of set-theoretic solutions of the Yang--Baxter equation. We introduce relative commuting probability for skew left braces and establish analogues of several classical structural results, including relative BFC-type theorems. For natural classes of finite skew left braces, we show that a positive lower bound for the commuting probability forces the existence of a large section which is trivial up to bounded subgroups. We also study the probability that two elements generate a trivial sub-skew brace and a Sylow-local version of commuting probability, obtaining further structural consequences. Finally, we introduce a commuting probability for finite non-degenerate set-theoretic solutions of the Yang--Baxter equation. Under natural hypotheses on the associated structure skew brace $G(X,r)$, a positive lower bound for this probability yields a bound on $|G(X,r):\Soc(G(X,r))|$ depending only on the probability and on $|X|$.

\end{abstract}

\maketitle

\section{Introduction}

Skew left braces provide an algebraic framework for the study of set-theoretic solutions of the Yang--Baxter equation. They were introduced by Guarnieri and Vendramin \cite{GV17}, extending the notion of a left brace introduced by Rump \cite{Rump07}, and have since become an important tool for translating questions about solutions of the Yang--Baxter equation into questions about groups equipped with two compatible operations.

A skew left brace $B=(B,+,\circ)$ consists of two group structures on the same set satisfying
$$
a\circ(b+c)=(a\circ b)-a+(a\circ c)
$$
for all $a,b,c\in B$. The common identity element of the two groups will be denoted by $1$. The interaction between the two operations is encoded by the maps
$$
\lambda_a(b)=-a+(a\circ b),
$$
and it is precisely this interaction that makes skew left braces useful in the study of the Yang--Baxter equation.

The aim of the present paper is to develop a probabilistic approach to this structure. Our starting point is commuting probability. In finite group theory, commuting probability and its relative versions have long been used to measure how close a finite group is to being abelian and to obtain quantitative structural information; see, for instance, \cite{ET68,WHG73,ERL07,DS22}. Commuting probability for finite skew left braces was recently introduced in \cite{MY26}. In the skew brace setting, however, commutativity simultaneously involves the two group operations and their interaction, so the corresponding probabilistic questions contain more information than the ordinary commuting probabilities of the two underlying groups.

For a finite skew left brace $B$, we define its commuting probability by
$$
\Pb(B)=\frac{1}{|B|^2}
|\{(x,y)\in B^2\mid x\circ y=x+y=y+x=y\circ x\}|.
$$
We introduce a relative version of this invariant. If $H$ is a sub-skew brace of $B$, we define
$$
\Pb(H,B)=\frac{1}{|H||B|}
|\{(x,y)\in H\times B\mid x\circ y=x+y=y+x=y\circ x\}|.
$$
More generally, the same definition gives $\Pb(H,K)$ for two sub-skew braces $H$ and $K$ of $B$.

Our first objective is to understand to what extent a positive commuting probability forces a finite skew left brace to contain a large part with simple structure. This is analogous to a classical theme in probabilistic group theory. We prove relative BFC-type results and show that, in several natural classes of skew left braces, positive commuting probability forces the existence of a large section which is trivial up to bounded subgroups.

The classes which occur naturally in our results are the two-sided, symmetric and $\lambda$-homomorphic skew left braces. We denote by $\mathcal T$ the class consisting of skew braces belonging to at least one of these three families. One of our main structural results states that if $B\in\mathcal T$ and $\Pb(B)$ is bounded away from zero, then $B$ contains ideals $B_2\leqslant B_1$ such that both $|B_2|$ and $|B:B_1|$ are bounded in terms of $\Pb(B)$, while $B_1/B_2$ is a trivial brace. Thus a positive commuting probability forces a bounded-by-trivial section of bounded index.

A second motivation for our approach comes directly from the Yang--Baxter equation. Let $(X,r)$ be a finite non-degenerate set-theoretic solution, written as
$$
r(x,y)=(\sigma_x(y),\tau_y(x)).
$$
We introduce its commuting probability by
$$
\operatorname{p}(X,r)
=
\frac{1}{|X|^2}
\big|\{(x,y)\in X^2\mid r(x,y)=(y,x)\text{ and }r(y,x)=(x,y)\}\big|.
$$
This invariant measures the proportion of pairs on which the solution behaves like the flip map in both directions. It is defined directly in terms of the solution and provides a probabilistic quantity on the Yang--Baxter side which can be compared with the commuting probability of the associated structure skew brace. For a recent probabilistic approach to finite non-degenerate solutions and skew braces, we refer to \cite{FMT26}.

One of the main applications of our structural results is that a positive lower bound for $\operatorname{p}(X,r)$ forces a strong restriction on the structure skew brace. More precisely, if $|X|=n$, the structure skew brace $G=G(X,r)$ belongs to $\mathcal T$, and $\operatorname{p}(X,r)\geqslant\epsilon>0$, then
$$
|G:\Soc(G)|
$$
is bounded in terms of $\epsilon$ and $n$ only. Thus probabilistic information involving only the original finite solution yields quantitative control on the structure skew brace modulo its socle.

We now state three of our main results. The first two describe structural consequences of commuting probability for skew left braces, while the third transfers this information to finite solutions of the Yang--Baxter equation.

\begin{thm} 
Let $H$ and $K$ be ideals of a finite skew left brace $B$. Suppose that, for every $x\in H\cup K$,
\[
|(H,\circ):\Cb_H(x)| \leqslant m \quad\text{and} \quad |(K,\circ):\Cb_K(x)| \leqslant m,
\]
where $\Cb_A(b)$ denotes the brace centralizer of an element $b \in B$ in the sub-skew brace $A$ of $B$.
Then $|H*K|$ and $|K*H|$ are $m$-bounded. In particular, taking $K = B$, we get $|[H, B]^b|$ is $m$-bounded, where $[H, B]^b$ denotes the brace commutator of $H$ and $B$.
\end{thm}

\begin{thm}\label{main-2}
Let $B$ be a finite skew left brace belonging to the class $\mathcal T$ such that $\Pb(B)\geqslant\epsilon>0$. Then there exist ideals $B_2 \leqslant B_1\trianglelefteq B$ such that $B_1/B_2$ is a trivial left brace, and $|B_2|$ and $|B:B_1|$ are bounded in terms of $\epsilon$ only.
\end{thm}

Conversely, the existence of such a bounded-by-trivial section gives a positive lower bound for the commuting probability. We also give an example showing that the preceding result does not hold in full generality for all finite skew braces.

\begin{thm}\label{main-3}
Let $(X,r)$ be a finite non-degenerate solution of size $n$ such that its structure skew brace $G:=G(X,r)$ lies in the class $\mathcal{T}$ and
$\operatorname{p}(X,r)\geqslant\epsilon>0$. Then
$|G:\Soc(G)|$ is bounded in terms of $\epsilon$ and $n$ only.
More precisely, set $\eta :=\eta_n(\epsilon)$ and
$L_n :=\operatorname{lcm}(1,\ldots,n)$. Let $M(\eta)$ and $N(\eta)$ be bounds for $|Q:Q_1|$ and $|Q_2|$, respectively, when Theorem~\ref{main-2} is applied to $Q=G/\Soc(G)$. Then
$$
|G :\Soc(G)|
\leqslant M(\eta)N(\eta)L_n^{2M(\eta)n}.
$$
\end{thm}

The paper also develops several further aspects of this probabilistic point of view. We establish basic inequalities for relative commuting probability and prove that it is invariant under a suitable notion of relative isoclinism. We also obtain structural consequences from equality between the commuting probabilities of a skew left brace and one of its sub-skew braces, as well as quotient-product inequalities analogous to familiar results in finite group theory. Two representative results are the following.

\begin{thm} 
Let $H$ be a sub-skew brace of a finite skew left brace $(B, +, \circ)$ such that $\Pb(B)=\Pb(H)$. Then $H$ is a left ideal of $B$, which is a normal subgroup of $(B,\circ)$. Moreover, if $B$ is a skew brace from the class $\mathcal{T}$, then $H$ is an ideal of $B$ and $B/H$ is a trivial brace.
\end{thm}

\begin{thm} 
Let $(B,+, \circ)$ be a skew left brace of finite order, $H$ and $K$ be its sub-skew braces and $I$ be its ideal. Then $\Pb(H,K)\leqslant \Pb(I\cap H, I\cap K)\Pb(HI/I,KI/I)$. In particular, $\Pb(H,K)\leqslant \Pb(I\cap H, I\cap K)$ and $\Pb(H,K) \leqslant \Pb(HI/I,KI/I)$.
\end{thm}

We also consider a different probability, denoted by $\Pt(B)$, measuring the proportion of pairs which generate a trivial sub-skew brace. This quantity detects the interaction between the two brace operations more directly. A positive lower bound for $\Pt(B)$ yields large additive and multiplicative pieces on which the lambda action becomes trivial modulo a bounded obstruction. We obtain quantitative bounds and examples showing that the principal estimates are essentially best possible. In particular, we prove the following.

\begin{thm} 
Let $B$ be a finite skew left brace with $\Pt(B)\geqslant\epsilon>0$, and put
$s=\lfloor1/\sqrt{\epsilon}\rfloor$. Then there exist
$A_0\unlhd(B,\circ)$ and strong left ideals $N\leqslant M_0\leqslant B$
such that
$$
|(B,\circ):A_0|\leqslant s,\quad
|(B,+):M_0|\text{ is $\epsilon$-bounded},\quad
|N|\text{ is $\epsilon$-bounded},\quad
A_0*M_0\leqslant N.
$$
Moreover, the induced lambda action on $M_0/N$ has image of order at most $s$.
\end{thm}

Finally, we introduce a Sylow-local commuting probability for finite skew left braces. This allows us to transfer recent local probabilistic results for finite groups to the brace setting and to obtain nilpotency and solubility criteria, as well as bounded-index structural results for two-sided skew braces. One of these results is as follows.

\begin{thm} 
Let $B$ be a finite two-sided skew left brace. If $\Pb_{\Syl}(B)\geqslant\epsilon>0$,
then there are canonical ideals $J \leqslant I$ of $B$ such that $[B:I]$ is
$\epsilon$-bounded, and both $J$ and $I/J$ are direct products of their
Sylow ideals.
\end{thm}

The paper is organized as follows. Section~2 contains the necessary background on skew left braces and set-theoretic solutions of the Yang--Baxter equation. In Section~3 we introduce relative commuting probability and relative isoclinism. Section~4 contains the BFC-type results and their structural consequences for the classes in $\mathcal T$. Section~5 is devoted to the probabilities $\Pt(B)$ and $P_\lambda(B)$ and to their quantitative structural consequences. In Section~6 we study Sylow-local commuting probability. The final section returns to the Yang--Baxter equation: we introduce the commuting probability of a finite non-degenerate solution and use the preceding results to obtain bounds on its structure skew brace modulo the socle.

We conclude this section by setting some notation. As mentioned above, the common additive and multiplicative identity of a skew left brace is denoted by $1$. For a skew left brace $(B, +, \circ)$ and $a, b \in B$, we set the commutator $[a, b]^+ := a+b-a-b$ in $(B, +)$ and the commutator $[a, b]^{\circ} := a\circ b \circ a^{-1} \circ b^{-1}$ in $(B, \circ)$. Left normed higher commutators in $(B, +)$ and $(B, \circ)$ are defined inductively. The center of a group $G$ is denoted by $\Z(G)$, and for any element $g \in G$, the centraliser of $g$ in $G$ is denoted by $\C_G(g)$.


\section{Preliminaries and Some Key Results}

A non-empty set $B$, equipped with two group structures  $(B,+)$ and $(B, \circ)$, is  said to be a {\it skew left brace} if for $a, b, c \in B$ the following compatibility condition (skew left distributivity) holds:
 $$a \circ (b + c) = a \circ b - a + a \circ c,$$
 where $-a$ is the inverse of $a$ in $(B, +)$. A {\it skew right brace} can be defined analogously, that is, demanding  skew right distributivity.
 A skew left brace $(B, +, \circ)$ is said to be a {\it two-sided skew brace} if $(B, +, \circ)$ is also  a skew right brace.   Note that if $(B, +, \circ)$ is a skew left brace with $(B, \circ)$ abelian, then $(B, +, \circ)$ is a two-sided skew brace.  
 
The groups $(B,+)$ and $(B,\circ)$ are, respectively, called  {\it additive group} and {\it multiplicative group} of the skew left brace $(B, +, \circ)$.  For simplicity of notation we denote a skew left brace $(B, +, \circ)$ by $B$. If we want to emphasise the role of  group operations, then we use the original notation.   We will normally drop the use of word `left' from skew left brace, and just say  skew brace, unless it is necessary to specify the property.

Every group $(B, +)$ can be viewed as a skew  brace by defining $\circ = +$. Such a skew brace is called \emph{trivial}. A skew brace which is not a trivial skew brace is called a \emph{non-trivial skew brace}.
 Easiest among non-trivial skew braces are skew braces $(B, +, +_{op})$ with $(B, +)$ any non-abelian group and $(B, +_{op})$  the opposite group of $(B, +)$, that is, $a +_{op} b = b + a$ for all $a, b \in B$.

   Let $X$ be a given property of  groups.   A skew brace $(B, +, \circ)$ is said to be {\it $X$-type} if its additive group $(B,+)$ satisfies  the property $X$. If $X$ is the abelian property of groups, then a  skew left brace $(B, +, \circ)$ is said to be  \emph{abelian-type}.  An abelian-type skew  brace is called a \emph{brace} in the literature. For simplicity of notation, we call an abelian-type trivial skew brace a `trivial brace'.
Let $\Aut(B, +)$ denote the automorphism group of the group $(B, +)$. For a skew brace  $(B, +, \circ)$ and its element  $a$, define a map $\lambda_a : (B, +) \to (B, +)$  by setting 
 $$\lambda_a(b) = - a + (a \circ b)$$ 
 for all $b \in B$. It is now well known that $\lambda_a$ is an automorphism of $(B, +)$ and the  map $\lambda : (B, \circ) \to \Aut(B, +)$; $a \mapsto \lambda_a$ is a group homomorphism.

  Let $(B,  +, \circ)$ be a skew  brace. A subset $H$ of $B$ is said to be a {\it sub-skew  brace} of the skew   brace $(B, +, \circ)$  if $(H, + , \circ)$ is also a skew  brace.   For a given subset $S$ of $B$, the sub-skew brace generated by $S$ is defined to be the smallest sub-skew brace of $B$ containing $S$.  A sub-skew brace $(I, +, \circ)$ of  $(B, +, \circ)$  is said to be a left ideal of $(B, +, \circ)$ if $\lambda_a(I) \subseteq I$ for all  $a \in B$.  A left ideal $(I, +, \circ)$ of $(B, +, \circ)$ is said to be an ideal of $(B, +, \circ)$ if $(I, +)$ and $(I, \circ)$ are normal subgroups of $(B, +)$ and $(B, \circ)$ respectively.   An ideal generated by a given subset $S$ of $B$ is defined to be the smallest ideal of $B$ containing $S$. It is easy to see that the  kernel of the homomorphism $\lambda$ is given by 
 $$\Ker(\lambda) = \{a \in B \mid a + b = a \circ b \mbox{ for all } b \in B \}.$$

For two skew braces $A$ and $B$, a map $f : A \to B$ is said to be a {\it brace homomorphism} if $f$ is a group homomorphism from the additive group of $A$ to the additive group of $B$ and also from the multiplicative group $A$ to the multiplicative group of $B$. For any brace homomorphism $f : A \to B$, it turns out that $\Ker(f)$, the kernel of $f$,  is an ideal of $A$. A bijective brace homomorphism is called a {\it brace isomorphism}.

Let $(B,+,\circ)$ be a skew brace and let $a,b\in B$. Define 
$$a \ast b := \lambda_a(b) - b = -a + (a \circ b) - b.$$
 The  {\it socle} of $B$,  denoted by $\Soc(B)$, is defined by 
 $$\Soc(B) := \Ker(\lambda) \cap \Z(B,+).$$
The  {\it annihilator} of $B$, denoted by $\Ann(B)$, is defined by 
 $$\Ann(B) := \Ker(\lambda) \cap Z(B,+) \cap \Z(B, \circ) = \Soc(B) \cap  \Z(B, \circ).$$
 It is not very difficult to see that  $\Ann(B)$  is an ideal of $B$. It  follows from \cite[Lemma 2.5]{GV17} or \cite[Proposition 2.9]{BEP24} that   $\Soc(B)$ is also an ideal of $B$. We now define higher annihilators of $B$. For $n \geqslant 2$, define $\Ann_n(B)$, the $n$th annihilator  of $B$, by
$$\Ann_n(B) / \Ann_{n-1}(B) = \Ann\big(B /  \Ann_{n-1}(B)\big).$$
This is possible because $\Ann_{n-1}(B)$ is an ideal of $B$ \cite[Definition 2.2]{JAV23}.  It turns out that 
$$\Ann_n(B) = \{a \in B \mid a*b, b*a,  [a, b]^{+}  \in \Ann_{n-1}(B) \mbox{ for all } b \in B\}.$$
This is an ascending series of ideals of $B$, which is called the  \textit{annihilator series}. We shall view this series as analogous to the upper central series of a group.

Let $B := (B, +, \circ)$ be a skew  brace and $\gamma_n(B, +)$ and  $\gamma_n(B, \circ)$, respectively, denote the $n$th terms of the lower central series of $(B, +)$ and $(B, \circ)$.

Throughout, for any two subsets $X$ and $Y$ of a skew brace $B$, $\{X*Y\}$  denotes the set $\{x*y: x \in X, y\in Y\}$ and  $X * Y$ denotes the subgroup of $(B, +)$ generated by $\{X*Y\}$. In general $X * Y$ has a very restricted structure even when $X$ and $Y$ are sub-skew braces of $B$. In particular, one should not expect $X*Y$ to be a sub-skew brace without additional hypotheses. We shall consider the situation when  one of the sets $X$ and $Y$ is a sub-skew brace  and the other is $B$ itself. So, let us specialize to this situation only.
 
Let $H$ be a sub-skew brace of a skew brace $B$. Denote by $[H, B]^b$ the sub-skew brace generated by the set $\{ H*B, B*H, [H, B]^+\}$. Note that  $[H, B]^b = [B, H]^b$. Considering $H$ to be an ideal of $B$, we can see
 
 \begin{lemma}
 Let $H$ be an ideal  of a skew brace $B$.  Then $B*H$ is a left ideal of $B$, and  $H*B$ and $[H, B]^b = \langle H*B, B*H, [H, B]^+\rangle^+$ are  ideals of $B$.
 \end{lemma}

 We can define two central series of a skew brace $(B, +, \circ)$ as follows:  Set $\Gamma_1(B)  := B$ and define
$$\Gamma_n(B) := \langle B * \Gamma_{n-1}(B), \Gamma_{n-1}(B) * B, [B, \Gamma_{n-1}(B)]^{+}\rangle^{+},$$
where $[B, \Gamma_{n-1}(B)]^{+}$ denotes the subgroup of $(B, +)$ generated by the set $\{[a, u]^{+} \mid a \in B, u \in \Gamma_{n-1}(B)\}$. We call this the lower central series of the skew brace $B$.

A skew brace $B$ is said to be  {\em nilpotent} if there exists an integer $n$ such that  $\Ann_n(B) = B$. The least such integer is called the  {\em nilpotency class} of $B$.  The authors of \cite{BJ23} call such skew braces centrally nilpotent and prove that   $\Ann_n(B) = B$ if and only if $\Gamma_{n+1}(B) = 1$ (\cite[Theorem 2.8]{BJ23}). 
 
Next, set $D_0(B) := B$ and for $n \geqslant 1$, define 
$$D_n(B):= \Gamma_2(D_{n-1}(B)).$$
Note that $D_1(B) = \Gamma_2(B)$. The series $\{D_n(B)\}_{n\geqslant 0}$ is called the derived series of the skew brace $B$. Following \cite{BEJP24}, we say that a skew brace $B$ is {\em solvable} or {\em soluble} if $D_n(B) = 1$ for some finite positive integer $n$. As in group theory, it is not difficult to see that if a skew brace $B$ is nilpotent, then it is solvable. 

Let $H$ be a sub-skew brace of a skew brace $B$. We define the  \emph{brace centralizer} of an element $x \in B$ in $H$ to be the subset 
\begin{eqnarray*}
\Cb_H(x) &:=&  \{b \in H \mid x * b = b * x = [x, b]^{+} = 1\}\\
&= &\{b \in H \mid x * b = [x, b]^{\circ} = [x, b]^{+} = 1\}.
\end{eqnarray*}
The {\it brace centralizer}  of a subset  $S$ of $B$ in $H$, denoted as $\Cb_{H}(S)$, is defined to be
$$\Cb_{H}(S) := \cap_{x \in S} \Cb_H(x).$$
 Note that if $x \in \Ann(H)$, then $\Cb_{H}(x) = H$. Also note that $\Cb_{H}(H) = \Ann(H)$. In particular, taking $H$ to be $B$, we see that $\Cb_{B}(B) = \Ann(B)$. It follows from \cite[Proposition 2.19]{SV18} that if $x \in \Cb_{B}(x)$ for all $x \in B$, then $B$ is a two-sided skew brace. The converse need not be true in general. 

 The following key result is proved in \cite[Theorem 3.3]{CFT25}.

\begin{pro} 
   Let $(B, +, \circ)$ be a skew brace.    If $x \notin \Ann(B)$, then $\Cb_{B}(x)$ is a proper subgroup of  $(B, \circ)$.
\end{pro}

A skew left brace $(B, +, \circ)$ is said to be \emph{symmetric} if $(B, \circ, +)$ is also a skew left brace. This notion was introduced in \cite{LC19}, and was further studied in \cite{BNY23, AC20}. Note that for a skew brace $(B, +, \circ)$ the map $(B, \circ) \to \Aut(B, +)$, $b \mapsto \lambda_b$ is also a map from $(B, +)$ to $\Aut(B, +)$.  We say that the skew brace $(B, +, \circ)$ is \emph{$\lambda$-homomorphic} if this map is also a group homomorphism from $(B, +)$ to $\Aut(B, +)$.  This concept was introduced in \cite{CCD20}, and was further explored in \cite{BNY22}. Let $\mathcal{T}$ denote the class of all skew braces that are either two-sided, symmetric, or $\lambda$-homomorphic.  We'll use this notation throughout the article.  We now observe

\begin{pro}\label{TS-S-L}
    Let $(B, +, \circ)$ lie in the class $\mathcal{T}$.  Then $\Cb_{B}(x)$ is a sub-skew brace of $B$.
\end{pro}

We conclude this section by defining set-theoretic solutions of the Yang--Baxter equation and recalling some basic facts.
Let $X$ be a non-empty set and $r : X \times X \to X \times X$, 
$r(x, y) = (\sigma_x(y), \tau_y(x))$, be a bijection. The pair $(X, r)$ is said to be a {\it set-theoretic solution of the Yang--Baxter equation} (YBE) or simply a {\it solution of YBE}, if
$$
r_1 r_2 r_1 = r_2 r_1 r_2,
$$
where $r_1 = r \times id$, $r_2 = id \times r$ are bijections of $X \times X \times X$. 
The solution $(X,r)$ is said to be non-degenerate if each $\sigma_x$ and each $\tau_y$ is a bijective map. The solution is said to be involutive provided $r^2 = \Id_{X \times X}$.

Etingof, Schedler and Soloviev \cite{ESS99} defined  the structure group as the group $(G(X, r), \circ)$  given by 
$$(G(X,r), \circ) := \gen{X \mid x \circ y = \sigma_x(y) \circ \tau_y(x) \mbox{ for } x, y \in X}.$$
Building on this, Soloviev \cite{S2000} and Lu, Yan and Zhu \cite{LYZ20} defined the derived structure group $(A(X,r),+)$ associated with the solution $(X,r)$ as follows:
$$(A(X,r), +) := \gen{X \mid x + \sigma_x(y) = \sigma_x(y) + \sigma_{\sigma_x(y)}(\tau_y(x)) \mbox{ for } x, y \in X}.$$
and proved the existence of a bijective  $1$-cocycle $\pi$ from $(G(X,r), \circ)$ to $(A(X,r), +)$, which gives rise to a skew brace $(A(X, r), + , \circ)$, where the group structure $(A(X, r), \circ)$ is determined by $\pi$. 

The following theorem \cite[Theorem 3.1]{GV17} connects skew braces to set-theoretic solutions.
\begin{pro} 
If $(B, +, \circ)$ is a skew  brace, then
\begin{equation*} 
r_B(a, b) = \left(\lambda_a(b), \lambda^{-1}_{\lambda_a(b)}(-(a \circ b) + a +(a \circ b)) \right)
\end{equation*}
is a non-degenerate set-theoretic solution of the Yang--Baxter equation.

If $(B, +, \circ)$ is a left brace, then
\begin{equation*} 
r_B(a, b) = \big(\lambda_a(b), \lambda^{-1}_{\lambda_a(b)}(a)\big).
\end{equation*}
Moreover, $r_B$~is involutive if and only if $a+b = b+a$ for all $a, b \in B$.
\end{pro}

  
  \section{Relative Commuting Probability and Relative Isoclinism}

  We start this section by defining  commuting probability of a finite skew  brace. For  a finite skew  brace  $B:= (B, +, \circ)$, following \cite{MY26}, we define the  \emph{commuting probability} of $B$, denoted by $\Pb(B)$, to be 
$$\Pb(B)  :=  \frac{1}{ |B|^{2}} \big(|\{(a, b) \mid a * b = b* a = [a, b]^+ = 1\}|\big).$$
It is clear that $\Pb(B) = 1$ for all  trivial braces $B$. For any other brace $B$,  $\Pb(B)$ lies in the open interval $(0, 1)$.
 Note that, analogous to group theory, we can derive $\Pb(B)$ in terms of brace centralizers as follows:
$$\Pb(B) = \frac{\sum_{x\in B}|\Cb_{B}(x)|}{|B|^2}.$$

A first observation is that if  $\Pb(B)  \ne  1$, then $B$ must admit a non-trivial proper sub-skew brace, which follows from the fact that a skew brace having no  non-trivial proper sub-skew brace is always a trivial brace of prime  order (\cite[Theorem A]{BEJP24}). The converse, of course, is not true as $\Pb(B) = 1$ for all trivial  braces. As a direct consequence of the definition, one can easily see that
$$\Pb(B) \leqslant \min\{\Pr(B, +), \Pr(B, \circ)\},$$
where $\Pr(G)$ denotes the commuting probability of a given group $G$.
  
We now introduce the  concept of the relative commuting probability for finite skew braces analogous to the same concept for finite groups. We start by recalling the definition for finite groups  from  \cite{ERL07}. The relative commuting probability of a subgroup $H$ in a finite group $G$,  denoted by $d(H, G)$, is  defined to be  the ratio
   $$d(H,G) := \frac{1}{|H||G|}|\{(x,y)\in H\times G \mid xy=yx\}|.$$
Now we prepare the setup for skew braces. Let $H$ be a sub-skew brace of a skew brace $B$.
Define $$\Ann(H,B) := \{x \in H \mid x \circ y = x + y = y + x = y \circ x~\mbox{ for all }  y \in B\}.$$
Note that $\Ann(H, B) = H \cap \Ann(B)$.  The following result is immediate. 
\begin{pro}
  For any skew brace $B$ and its sub-skew brace $H$,    $\Ann(H,B)$ is an ideal of $B$.
\end{pro}

 Let  $B$ be a skew brace  and $H$ be its sub-skew brace.    The \emph{relative commuting probability} of the pair  $(H, B)$, denoted by  $\Pb(H,B)$, is defined by the ratio
   $$\Pb(H,B) :=\frac{1}{|H||B|}|\{(x,y)\in H\times B \mid  x \circ y = x + y = y + x = y \circ x\}|.$$
  In general,  we can define relative commuting probability of any pair  $(H, K)$ of sub-skew braces $H$ and $K$ of a skew brace $B$, denoted by  $\Pb(H,K)$, by
       $$\Pb(H,K) :=\frac{1}{|H||K|}  |\{(x,y) \in H \times K \mid x \circ y = x + y = y + x = y \circ x\}|.$$
The following result is immediate   from the definition, but we mention a brief proof.    

\begin{lem} 
    Let $B$ be a finite skew brace and $H$ be its sub-skew brace. Then   $\Pb(B)\leqslant \Pb(H,B)\leqslant \Pb(H)$.
\end{lem}
\begin{proof}
We apply the definition in two different equivalent forms. First by taking centralizers in $H$, we get
      \begin{align*}
            \Pb(H,B)&=\frac{1}{|H||B|}\sum_{x\in B}|\Cb_{H}(x)| = \frac{1}{|B|}\sum_{x\in B}\frac{|\Cb_{H}(x)|}{|H|}\\
                     &\geqslant \frac{1}{|B|}\sum_{x\in B}\frac{|\Cb_{B}(x)|}{|B|} =\frac{1}{|B|^2}\sum_{x\in B}|\Cb_{B}(x)|\\
                     &=\Pb(B).
      \end{align*}
 Now we take centralizers in $B$ and get
      \begin{align*}
            \Pb(H,B)&=\frac{1}{|H||B|}\sum_{y\in H}|\Cb_{B}(y)| =\frac{1}{|H|}\sum_{y\in H}\frac{|\Cb_{B}(y)|}{|B|}\\
                     &\leqslant\frac{1}{|H|}\sum_{y\in H}\frac{|\Cb_{H}(y)|}{|H|}=\Pb(H). 
      \end{align*}
      The proof is complete.
\end{proof}

As an easy consequence, we get
\begin{cor}
    Let $B$ be a finite skew brace and $H$ be its sub-skew brace. Then $\Pb(B)=\Pb(H)$ if and only if
    $$\Pb(B)=\Pb(H,B)=\Pb(H).$$
\end{cor}

The sub-skew braces $H$ of a skew brace $B$ such that $\Pb(B)=\Pb(H)$ enjoy nice properties as shown in the following result.

\begin{thm} 
    Let $H$ be a sub-skew brace of a finite skew brace $B$ such that $\Pb(B)=\Pb(H)$. Then $H$ is a left ideal of $B$, which is a normal subgroup of $(B,\circ)$. Moreover, if $B$ is a skew brace from the class $\mathcal{T}$, then  $H$ is an ideal of $B$ and $B/H$ is a trivial brace.
\end{thm}
\begin{proof}   
       By the given hypothesis, we get
       $$ \Pb(B)=\frac{\sum_{x\in B}|\Cb_B(x)|}{|B|^2}= \frac{{|(B,\circ):(H,\circ)|}^2\sum_{y\in H}|\Cb_H(y)|}{|B|^2}=\Pb(H), $$
       which, by a direct comparison, gives
       $$\sum_{x\in B}|\Cb_B(x)|={|(B,\circ):(H,\circ)|}^2\sum_{y\in H}|\Cb_H(y)|.$$
       If possible, let there exist an element $x\in B$ such that $|\Cb_B(x)|<|(B,\circ):(H,\circ)||\Cb_H(x)|$. Then 
       \begin{align*}
           \sum_{x\in B}|\Cb_B(x)|&<{|(B,\circ):(H,\circ)|}\sum_{x\in B}|\Cb_H(x)|\\
                                  &={|(B,\circ):(H,\circ)|}\sum_{y\in H}|\Cb_B(y)|\\
                                  &\leqslant{|(B,\circ):(H,\circ)|}^2\sum_{y\in H}|\Cb_H(y)|,
       \end{align*}
       a contradiction. Therefore  $|\Cb_B(x)|=[(B,\circ):(H,\circ)]|\Cb_H(x)|$ for all $x\in B$. Now, for any element $x\in H$, we see that
           $$
              | H\circ \Cb_B(x) |=\frac{|\Cb_B(x)||H|}{|\Cb_H(x)|}=\frac{|B:H||\Cb_H(x)||H|}{|\Cb_H(x)|}=|B|.
           $$
           So $B=H\circ \Cb_B(x)$ for all $x\in H$. Let $b\in B$ and $h\in H$ be arbitrary elements. Note that $b=h_1\circ c$ for some $h_1\in H$ and $c\in \Cb_ B(h)$ (taking $x = h$). Then 
           $$\lambda_{b}(h)=\lambda_ {h_ 1\circ c}(h)=\lambda_{h_1}(\lambda_c(h))=\lambda_ {h_ 1}(h) \in H,$$
           which shows that $H$ is a left ideal of $B$. That $(H, \circ)$ is a normal subgroup of $(B,\circ)$ is easy to see.

        Now assume that $B$ lies in the class $\mathcal T$.
           Then, as above, it follows that  $B = H + \Cb_B(x)$ for all $x \in H$, and therefore  $H$ is an ideal of $B$. By \cite{MY26}, $\Pb(B)\leqslant \Pb(H)\Pb(B/H)$, which implies  $\Pb(B/H)=1$. Hence $B/H$ is a trivial brace, and the proof is complete.           
\end{proof}

\begin{lem}
      Let $H_1$ and $H_2$ be  sub-skew braces of a skew brace $B$ such that $H_1\subset H_2$. Then 
      $\Pb(H_2,B)\leqslant \Pb(H_1,B)\leqslant \Pb(H_1,H_2)$.
\end{lem}
\begin{proof}
      Since $H_1\leqslant H_2\leqslant B$, it is easy to see that
      $$[H_1:\Cb_{H_1}(x)]\leqslant~[H_2:\Cb_{H_2}(x)]\leqslant~[B:\Cb_{B}(x)]$$
      for all $x\in B$. Now, by the definition of the relative commuting probability we get
      \begin{align*}
            \Pb(H_1,H_2)&=\frac{1}{|H_1||H_2|}\sum_{x\in H_1}|\Cb_{H_2}(x)| =\frac{1}{|H_1|}\sum_{x\in H_1}\frac{|\Cb_{H_2}(x)|}{|H_2|}\\
                     &\geqslant \frac{1}{|H_1|}\sum_{x\in H_1}\frac{|\Cb_{B}(x)|}{|B|} =\frac{1}{|H_1||B|}\sum_{x\in H_1}|\Cb_{B}(x)|\\
                     &=\Pb(H_1,B).
      \end{align*}
      Similarly,
      $$\Pb(H_1,B)  =  \frac{1}{|B|}\sum_{x\in B}\frac{|\Cb_{H_1}(x)|}{|H_1|} \geqslant \frac{1}{|B|}\sum_{x\in B}\frac{|\Cb_{H_2}(x)|}{|H_2|} = \Pb(H_2,B).$$
      This completes the proof.
\end{proof}

Here is another result for arbitrary sub-skew braces.
\begin{thm} 
    Let $(B,+, \circ)$ be a skew brace of finite order,  $I$ be an ideal of $B$ and $H,K$ be sub-skew braces of $B$. Then $\Pb(H,K)\leqslant \Pb(I\cap H, I\cap K)\Pb(HI/I,KI/I)$. In particular, $\Pb(H,K)\leqslant \Pb(I\cap H, I\cap K)$ and $\Pb(H,K) \leqslant \Pb(HI/I,KI/I)$.
\end{thm}
\begin{proof}
     For simplicity of notation we suppress the use of `$\circ$' and write $x \circ I = xI$ for a left coset of $I$ in $(B, \circ)$. For the product of two subgroups $H_1$ and $H_2$  of $(B, \circ)$ we simply write $H_1 H_2$.
 We first prove that $({\Cb_{K}(x)  I)/I} \subset \Cb_{KI/I}(x  I)$ for all $x\in H$. Let $y \in \Cb_{K}(x)$. Since $I$ is an ideal of $B$, $H\cap I$ is an ideal of $H$. So $x + (H\cap I) = x \circ (H \cap I)$ for all $x \in H$. By an easy computation, we get
 \begin{eqnarray*}
(y  I) * (x  I) &=& (y*x)  I  = I,\\
  (x  I) * (y  I) &=& (x*y)  I  = I  
  \end{eqnarray*}
  and 
  $$[x I, y  I]^{\circ} = [x, y]^{\circ}  I  = I.$$
Hence $({\Cb_{K}(x)  I)/I} \subset \Cb_{KI/I}(x  I)$.

Now, using the definition of $\Pb(H,K)$, we have
\begin{align*}
  |H||K|\Pb(H,K) &=\sum_{x\in H}|\Cb_{K}(x)| =\sum_{b (H\cap I)\in{H/(H\cap I)}} ~ \sum_{x\in b (H\cap I)}|\Cb_{K}(x)|\\
             &=\sum_{b (H\cap I)\in{H/(H\cap I)}} ~ \sum_{x\in b (H\cap I)}|{(\Cb_{K}(x) I})/I||\Cb_{K\cap I}(x)|\\
             &(\Cb_{K}(x) I)/I \cong_{(\text{w.r.t. }\circ)} \Cb_K(x)/(\Cb_K(x) \cap I) = \Cb_K(x)/\Cb_{K\cap I}(x)\\
             &\leqslant \sum_{b (H\cap I)\in{H/(H\cap I)}} ~ \sum_{x\in b (H\cap I)}|\Cb_{KI/I}(x I)| |\Cb_{K\cap I}(x)|\\  
             &= \sum_{b (H\cap I)\in{H/(H\cap I)}}\Big(|\Cb_{KI/I}(b (H\cap I))|\sum_{x\in b (H\cap I)}|\Cb_{K\cap I}(x)|\Big)\\   
             &= \sum_{b (H\cap I)\in{H/(H\cap I)}}\Big(|\Cb_{KI/I}(b (H\cap I))| \sum_{y\in (K\cap I)}|\Cb_{b (H\cap I)}(y)|\Big).\\                      
      \end{align*}
      Let $z \in   \Cb_{b (H\cap I)}(y)$.  Then $z  (H\cap I) = b (H\cap I)$ as $z \in b  I$. Now, since $\Cb_{B}(y)$ is a subgroup of $(B, \circ)$ and $z \in \Cb_B(y)$, we get
    \begin{eqnarray*}
        \Cb_{b (H\cap I)}(y) &=&\Cb_{B}(y)\cap b (H\cap I) = z \Cb_{B}(y) \cap z (H\cap I) = z \big(\Cb_{B}(y)\cap (H\cap I)\big)\\
        &=& z   \Cb_{(H\cap I)}(y),
        \end{eqnarray*}
     which gives $|\Cb_{b (H\cap I)}(y)| = |\Cb_{(H\cap I)}(y)|$. Substituting these values into the preceding inequality, we get
    \begin{align*}
        |H||K|\Pb(H,K)  &\leqslant \sum_{b (H\cap I)\in{H/(H\cap I)}}\Big(|\Cb_{KI/I}(b (H\cap I))| \sum_{y\in (K\cap I)}|\Cb_{(H\cap I)}(y)|\Big)\\
        &= \sum_{b (H\cap I)\in{H/(H\cap I)}}\big(|\Cb_{KI/I}(b (H\cap I))|\big) \sum_{y\in (K\cap I)}\big(|\Cb_{(H\cap I)}(y)|\big),
    \end{align*}
    The right-hand side is equal to
    $$
    |H||K|\,\Pb(HI/I,KI/I)\Pb(H\cap I,K\cap I),
    $$
    and the result follows.
\end{proof}

 We now establish a bound on  $\Pb (H,B)$.
      
\begin{pro} 
Let $H$ be a sub-skew brace of a finite skew brace $B$. Then    $\Pb (H,B)\leqslant \frac{|\Ann(B) \cap H|+|H|}{2|H|}.$
\end{pro}
\begin{proof}
 We use the definitions of the commuting probability and $\Ann(H, B)$ to get 
      \begin{align*}
            {|H||B|}\Pb(H,B)&=\sum_{y\in H}|\Cb_{B}(y)|\\
                             &=\sum_{y\in \Ann(H,B)}|\Cb_{B}(y)|+\sum_{y\in (H- \Ann(H,B))}|\Cb_{B}(y)| \\
                             &=|\Ann(H,B)||B|+\sum_{y\in (H-\Ann(H,B))}|\Cb_{B}(y)| \\
                             &\leqslant |\Ann(H,B)||B|+|H-\Ann(H,B)|\frac{|B|}{2}\\
                             & =\frac{|\Ann(H,B)||B|+|H||B|}{2}.
      \end{align*}
  From this we get $\Pb(H,B)\leqslant \frac{|\Ann(H,B)|+|H|}{2|H|}$.
\end{proof}

As a consequence we get
\begin{cor}
   Let $H$ be a sub-skew brace of a finite skew brace $B$ not contained in $\Ann(B)$. Then  $\Pb (H,B)\leqslant \frac{3}{4}$. Moreover, if $\Pb (H,B)= \frac{3}{4}$, then $H/{\Ann(H,B)}\cong \mathbb{Z}_2$.
\end{cor}
\begin{proof}
    As $H\nsubseteq Ann(B)$, $\Ann(B)\cap H$ is a proper subgroup of $H$. Thus $|\Ann(H, B)| = |\Ann(B)\cap H| \leqslant \frac{|H|}{2}$. By the preceding result we get  $\Pb (H,B)\leqslant \frac{\frac{|H|}{2}+|H|}{2|H|}=\frac{3}{4}$.

    Let   $\Pb(H,B) = \frac{3}{4}$. Then  it follows, again by the preceding result, that $\frac{|H|}{|\Ann(H,B)|} \leqslant 2$. The case $\frac{|H|}{|\Ann(H,B)|} = 1$ is possible only when $H = \Ann(H, B)$, which implies that $H \subseteq \Ann(B)$; this is impossible by the hypothesis.   Hence 
      $\frac{|H|}{|\Ann(H,B)|}=2$, which gives $H/{\Ann(H,B)}\cong \mathbb{Z}_2$, as there exists only one skew brace $\mathbb{Z}_2$ of order $2$.
\end{proof}

We now turn to the definition of isoclinism for pairs of skew braces. The analogous concept in group theory was defined in \cite{ERR13}. Let $H_i$ be a subgroup of $G_i$ for $i\in\{1,2\}$.  Define $\Z{(H_i,G_i)}:=H_i\cap Z(G_i)$ and $(H_i,G_i)^\prime := [H_i,G_i]$, the subgroup of $G_i$ generated by the set $\{[h_i, g_i] \mid h_i\in H_i,g_i\in G_i\}$. Define maps  
$${\phi^{(H_i,G_i)}} : H_i/\Z{(H_i,G_i)} \times G_i/\Z(G_i) \to  (H_i,G_i)^\prime$$
given by
$${\phi^{(H_i,G_i)}}(\bar{h}_i,\bar{g}_i)= [h_i,g_i]$$
for all $\bar{h}_i=h_iZ{(H_i,G_i)}\in {H_i/Z{(H_i,G_i)}}$ and  $\bar{g_i}=g_iZ{(G_i)}\in {G_i/Z{(G_i)}}$. It is easy to see that the maps $\phi^{(H_i,G_i)}$ are well-defined.

The pair of groups $(H_1,G_1)$ is said to be isoclinic to the pair of groups $(H_2,G_2)$ if there are two group isomorphisms $\alpha:G_1/Z{(G_1)}\rightarrow G_2/Z{(G_2)}$ and $\beta:(H_1,G_1)^\prime\rightarrow(H_2,G_2)^\prime$ such that the restriction $\alpha^\prime:=\alpha{\restriction}_{H_1/Z{(H_1,G_1)}}$ is an isomorphism from $H_1/Z{(H_1,G_1)}$ onto $H_2/Z{(H_2,G_2)}$, and the following diagram commutes:
\begin{center}
      \begin{tikzcd}
            {H_1/Z{(H_1,G_1)}}\times {G_1/Z{(G_1)}}\arrow{r}{\phi^{(H_1,G_1)}}\arrow{d}{\alpha^\prime\times \alpha} &  (H_1,G_1)^\prime\arrow{d}{\beta} \\
            {H_2/Z{(H_2,G_2)}}\times {G_2/Z{(G_2)}}\arrow{r}{\phi^{(H_2,G_2)}}& (H_2,G_2)^\prime.
      \end{tikzcd}
\end{center}

We now prepare for the definition of relative isoclinism for skew braces. For a skew brace $B$ and its sub-skew brace $H$, we set $\overline{B}= B/\Ann(B)$ and $\overline{H}:= H/\Ann(H,B)$. We also put $(H,B)^\prime=[H,B]^b$. We start with the following important observation.
\begin{lem} 
    Let $H$ be a sub-skew brace of a skew brace $B$, let $C$ be any skew brace and $f: \overline{B} \rightarrow C$ be a brace homomorphism. Then the map $f^\prime :B/\Ann(H,B) \to C$, given by $f^\prime({b}^\prime)=f(\overline{b})$
    is a well-defined brace homomorphism, 
    where ${b}^\prime=b+\Ann(H,B)$ and $\overline{b}=b+ \Ann(B)$. 
\end{lem}
\begin{proof}
    Let $b_1+\Ann(H, B)=b_2+\Ann(H, B)$. Then $b_1-b_2\in \Ann(H, B)\subseteq \Ann(B)$. So $f^\prime({(b_1-b_2)}^\prime)=f(\overline{b_1-b_2})=f(1+\Ann(B))=1_C$. This implies $1_C=f(\overline{b_1-b_2})=f(\overline{b_1})-f(\overline{b_2})=f^\prime(b_1^\prime)-f^\prime (b_2^\prime)\Rightarrow f^\prime(b_1^\prime)=f^\prime (b_2^\prime)$. Then $f^\prime$ is well-defined. As $f$ is a homomorphism, $f^\prime$ is a homomorphism. The proof is complete.
\end{proof}

\begin{pro}
      The following three maps 
      $${\phi_\ast^{(H,B)}},{\phi_\ast^{(B,H)}},{\phi_+^{(H,B)}}:{\overline{H}}\times {\overline{B}}\rightarrow(H,B)^\prime$$
      given by 
      $${\phi_\ast^{(H,B)}}(\bar{h},\bar{b})=h\ast b, \quad {\phi_\ast^{(B,H)}}(\bar{h},\bar{b})=b\ast h \mbox{ and } {\phi_+^{(H,B)}}(\bar{h},\bar{b})=[h,b]^+$$ 
      for all $\bar{h}=h \Ann{(H,B)}\in {\overline{H}}$ and  $\bar{b}=b \Ann{(B)}\in {\overline{B}}$  are well-defined. 
\end{pro}
\begin{proof}
      Let $\bar{h} = \bar{h}^\prime \in {\overline{H}}$ and $\bar{b} = \bar{b}^\prime \in {\overline{B}}$. Since $h$ and $h'$ as well as $b$ and $b'$ differ by an element of $\Ann(B)$, it is easy to see that
     ${\phi_+^{(H,B)}}$ is well-defined. We now only prove that 
  ${\phi_\ast^{(H,B)}}$ is well-defined, and the other case follows by symmetry. Since $h$ and $h'$ differ by an element of $\Ann(B)$, we get 
$$
h*b = \lambda_h(b) - b = \lambda_{h'}(b) - b.
$$
Note that $b = b' +u$ for some $u \in \Ann(B)$ as $b$ and $b'$ differ by an element of $\Ann(B)$. Thus 
$$h*b = \lambda_{h'}(b) - b = \lambda_{h'}(b') - b' = h' * b'$$
and the proof is complete.
\end{proof}

We are now ready to define relative isoclinism.
      Let $H_i$ be a sub-skew brace of the skew brace $B_i$ for $i = 1, 2$. 
      The pair $(H_1,B_1)$ is said to be isoclinic to the pair $(H_2,B_2)$
 if there are two isomorphisms $\xi:\overline{B}_1\rightarrow \overline{B}_2$ and $\theta:(H_1,B_1)^\prime\rightarrow (H_2,B_2)^\prime$ such that the restriction $\xi^\prime:=\xi{\restriction}_{\overline{H}_1}$ is an isomorphism from $\overline{H}_1$ onto $\overline{H}_2$, and the following diagrams commute:
\begin{center} 
      \begin{tikzcd}
            (H_1,B_1)^\prime\arrow{dd}{\theta} && {H_1/D_1}\times {B_1/C_1}\arrow{rr}{\phi_\ast^{(H_1,B_1)}}\arrow{ll}{\phi_+^{(H_1,B_1)}}\arrow{dd}{\xi^\prime\times \xi} \arrow[bend right,swap]{rr}[description]{\phi_\ast^{(B_1,H_1)}}&&  (H_1,B_1)^\prime\arrow{dd}{\theta}\\ \\
            (H_2,B_2)^\prime && {H_2/D_2}\times {B_2/C_2}\arrow{rr}{\phi_\ast^{(H_2,B_2)}}\arrow{ll}{\phi_+^{(H_2,B_2)}} \arrow[bend right,swap]{rr}[description]{\phi_\ast^{(B_2,H_2)}} && (H_2,B_2)^\prime,\\
   
      \end{tikzcd}
\end{center}
where $C_1:=\Ann{(B_1)}$, $D_1 := \Ann{(H_1,B_1)}$, $C_2:=\Ann{(B_2)}$ and $D_2:= \Ann{(H_2,B_2)}$. The pair $(\xi,\theta)$ is called an isoclinism between the pairs $(H_1,B_1)$ and $(H_2,B_2)$. By taking $H_1 = B_1$ and $H_2 = B_2$, we get the definition of isoclinism between the skew braces $B_1$ and $B_2$, which was introduced in \cite{LV24}.

 We now prove that relative commuting probability is invariant under the relative  isoclinism defined above.
 \begin{thm}\label{relative isoclinism and commuting probability}
       If the pair of skew braces $(H_1,B_1)$ is isoclinic to the pair of skew braces $(H_2,B_2)$, then $\Pb(H_1,B_1) = \Pb(H_2,B_2)$.
 \end{thm}
 \begin{proof}
       Let $(\xi,\theta)$ be an isoclinism between $(H_1,B_1)$ and $(H_2,B_2)$. Set 
       $$P_1 := |\overline{H}_1||\overline{B}_1|\Pb(H_1,B_1) \mbox{ and } P_2:= |\overline{H}_2||\overline{B}_2|\Pb(H_2,B_2).$$ 
       Now
      \begin{align*}
        P_2  &=\frac{|\{(x,y)\in H_2\times B_2 \mid [x,y]^+=x\ast y=y\ast x=1\}|}{|\Ann(H_2,B_2)||\Ann(B_2)|}\\
            &=\frac{|\{(x,y)\in H_2\times B_2 \mid \phi_+^{(H_2,B_2)}(\bar{x},\bar{y})=\phi_\ast ^{(H_2,B_2)}(\bar{x},\bar{y})=\phi_\ast ^{(B_2,H_2)}(\bar{y},\bar{x})=1\}|}{|\Ann(H_2,B_2)||\Ann(B_2)|}\\
            &=|\{(\bar{x},\bar{y})\in \overline{H}_2 \times \overline{B}_2\mid \phi_+^{(H_2,B_2)}(\bar{x},\bar{y})=\phi_\ast ^{(H_2,B_2)}(\bar{x},\bar{y})=\phi_\ast ^{(B_2,H_2)}(\bar{y},\bar{x})=1\}|.
      \end{align*}

      Similarly,
      $$P_1=|\{(\bar{x},\bar{y})\in \overline{H}_1\times \overline{B}_1\mid \phi_+^{(H_1,B_1)}(\bar{x},\bar{y})=\phi_\ast ^{(H_1,B_1)}(\bar{x},\bar{y})=\phi_\ast ^{(B_1,H_1)}(\bar{y},\bar{x})=1\}|.$$
      Now using the fact that $\theta$ is an isomorphism and the preceding diagram commutes, we get
      \begin{align*}
            P_1  &=|\{(\bar{x},\bar{y})\in \overline{H}_1 \times \overline{B}_1 \mid \theta \phi_+^{(H_1,B_1)}(\bar{x},\bar{y})=\theta \phi_\ast ^{(H_1,B_1)}(\bar{x},\bar{y})=\theta \phi_\ast ^{(B_1,H_1)}(\bar{y},\bar{x})=1\}|\\
            &=|\{(\bar{x},\bar{y})\in \overline{H}_1\times \overline{B}_1\mid\phi_+^{(H_2,B_2)}(\xi,\xi)(\bar{x},\bar{y})=\phi_\ast ^{(H_2,B_2)}(\xi,\xi)(\bar{x},\bar{y})  =\phi_\ast ^{(B_2,H_2)}(\xi,\xi)(\bar{y},\bar{x})\\
            & \quad \quad \quad \quad \hspace{1in}=1\}|\\
            &=|\{(\bar{x}^\prime,\bar{y}^\prime)\in \overline{H}_2\times \overline{B}_2\mid \phi_+^{(H_2,B_2)}(\bar{x}^\prime,\bar{y}^\prime)=\phi_\ast ^{(H_2,B_2)}(\bar{x}^\prime,\bar{y}^\prime)=\phi_\ast ^{(B_2,H_2)}(\bar{y}^\prime,\bar{x}^\prime)=1\}|\\
            & = P_2,
  \end{align*}
 where $\xi (\bar{y})=\bar{y}^\prime,\xi (\bar{x})=\bar{x}^\prime$. 
By the definition of isoclinism, $\overline{B}_1\cong \overline{B}_2$ and $\overline{H}_1 \cong \overline{H}_2$, which completes the proof. 
 \end{proof}

\begin{lem}\label{sub-skew brace isoclinic}
    Let $A$ and $B$ be two isoclinic skew braces and $H_A$ be a sub-skew brace of $A$. Then there is a sub-skew brace $H_B$ of $B$ such that $(H_A,A) $ is isoclinic to $(H_B,B)$.
\end{lem}
\begin{proof}
    Let $(\xi,\theta)$ be an isoclinism between $A$ and $B$, and define the sub-skew brace $H_B$ by
    $$
    \xi\big((H_A+\Ann(A))/\Ann(A)\big)=H_B/\Ann(B).
    $$
    Since $H_B$ contains $\Ann(B)$, we have $\Ann(H_B,B)=\Ann(B)$. Moreover,
    $$
    (H_A+\Ann(A))/\Ann(A) \cong H_A/(H_A\cap\Ann(A))
    =H_A/\Ann(H_A,A).
    $$
    Thus the restriction of $\xi$ induces an isomorphism
    $$
    \xi':H_A/\Ann(H_A,A)\longrightarrow H_B/\Ann(H_B,B).
    $$
    The defining commutative diagrams for $(\xi,\theta)$ show that
    $\theta((H_A,A)^\prime)=(H_B,B)^\prime$. Hence
    $$
    \theta'=\theta{\restriction}_{(H_A,A)^\prime}:(H_A,A)^\prime\longrightarrow(H_B,B)^\prime
    $$
    is an isomorphism, and the required diagrams for relative isoclinism commute. Therefore $(H_A,A)$ is isoclinic to $(H_B,B)$.
\end{proof}

\begin{cor}\label{existance of relative stem}
    Let $H_A$ be a sub-skew brace of a finite skew brace $A$. Then there is a skew brace $B$ and a sub-skew brace $H_B$ of $B$ such that $(H_A,A) $ is isoclinic to $(H_B,B) $ and $\Ann(H_B,B)\subseteq H_B\cap \Gamma_2(B)$. Moreover, if $A$ is finite, then $B$ is also finite.
\end{cor}
\begin{proof}
    By \cite[Theorem 2.18]{LV24}, there exists a skew brace $B$ isoclinic to the skew brace $A$ such that $\Ann(B)\subseteq \Gamma_2(B)$. By Lemma \ref{sub-skew brace isoclinic}, there is a sub-skew brace $H_B$ of $B$ such that $(H_A,A) $ is isoclinic to $(H_B,B)$. It easily follows that $\Ann(H_B,B)\subseteq H_B\cap \Gamma_2(B)$. The finiteness of $B$ also follows from \cite{LV24}.  The proof is complete. 
\end{proof}

\begin{cor}
    Let $H_A$ be a sub-skew brace of a finite skew brace $A$. Then there is a skew brace $B$ and a sub-skew brace $H_B$ of $B$ such that $\Pb(H_A,A) = \Pb(H_B,B) $ and $\Ann(H_B,B)\subseteq H_B\cap \Gamma_2(B).$
\end{cor}

\begin{proof}
    The proof follows from Corollary \ref{existance of relative stem} and Theorem \ref{relative isoclinism and commuting probability}.
\end{proof}

The pair $(H_B,B)$ in Corollary \ref{existance of relative stem} is called a \emph{stem pair} of skew braces. We close this section with the remark that for studying relative commuting probability, it is sufficient to study a stem pair of skew braces in its isoclinism class.

\section{ BFC-Type Theorems}
 
In this section we consider two themes. Under the first one, we prove an analogue of the group-theoretic relative BFC theorem for skew braces. Recall that $\mathcal{T}$ denotes the class of all skew braces that are either two-sided, symmetric, or $\lambda$-homomorphic. Under the second theme, we establish, for the skew braces $B$ lying in the class $\mathcal{T}$ and having positive commuting probability $\Pb(B)$, the existence of ideals $B_2 \leqslant B_1 \normaleq B$ such that $B_1/B_2$ is a trivial brace, and $|B_2|$ and $|B:B_1|$ are bounded in terms of $\Pb(B)$ only.  

 We start with the following interesting lemma by S. Eberhard \cite{SE15} for groups:
\begin{lem}\label{Eb-lem}
 Let $G$ be a finite group and $X$ a symmetric subset of $G$ containing the identity element. If there exists an integer $r$ such that  $(r + 1)|X|>|G|$, then $\langle X\rangle=X^{3r}$.
\end{lem}

We now prove Theorem A.

\begin{thm}\label{lem:relative-BFC}
Let $H$ and $K$ be ideals of a finite skew brace $B$. Suppose that, for every $x\in H\cup K$,
\[
|(H,\circ):\Cb_H(x)| \leqslant m \quad\text{and} \quad |(K,\circ):\Cb_K(x)| \leqslant m.
\]
Then $|H*K|$ and $|K*H|$ are $m$-bounded. In particular, taking $K = B$, we get that $|[H, B]^b|$ is $m$-bounded.
\end{thm}
\begin{proof}
We prove the assertion for $H*K$; the proof for $K*H$ is symmetric. First we bound the mixed set $\{H*K\}$. Choose $a\in K$ so that $|H*a|$ is maximal. Since $\Cb_H(a)$ has index at most $m$ in $(H,\circ)$, we have $|H*a|\leqslant m$. Write $H*a=\{h_1*a,\dots,h_s*a\}$ with $s\leqslant m$, and set  $C :=\bigcap_{i=1}^{s}\Cb_K(h_i)$. Then $|(K,\circ):C|\leqslant m^m$. If $c\in C$, then $h_i*c=1$ for every $i$, and hence
\[
h_i*(a+c)=h_i*a+a+h_i*c-a=h_i*a.
\]
By the maximality of $|H*a|$, it follows that $H*(a+c)=H*a$ for every $c\in C$. Consequently, for every $h\in H$ and $c\in C$,
\[
h*c\in -a-(H*a)+(H*a)+a,
\]
and therefore $|\{H*C\}|\leqslant m^2$.

Let $D =\langle C\rangle^+$ be the subgroup of $(K,+)$ generated by $C$. It now follows that $|(K,+):D|$ is $m$-bounded. Indeed, $|(K, \circ):C|$ is $m$-bounded and $|D|$ divides $|K|$.
 We also note that the additive commutator subgroup $[H, K]^+$ of $(B, +)$ (also a subgroup of $H$ as well as of $K$) has $m$-bounded order by the relative BFC theorem for groups; see, for instance, \cite{DS22}, because the brace centralizer is contained in the additive centralizer and the bounds in the hypothesis hold in both $(H,+)$ and $(K,+)$. This also proves that the additive  commutator subgroups $[H, H]^+$ and $[K, K]^+$ of $(H, +)$ and $(K, +)$, respectively, are of $m$-bounded orders.

Since $|C| \leqslant |D| \leqslant |K|$, there exists an $m$-bounded integer $s$ such that $(s+1)|X|>|(D,+)|$, where $X=C\cup-C$. Put $r=3s$. Hence, by Lemma~\ref{Eb-lem}, every element of $D$ is a sum of at most $r$ elements from $X$. More precisely, an arbitrary element  $d\in D$ can be written in the form $d=c_1+\dots+c_r$, where $c_i \in X$. Write 
$$\{H*D\} = \{H*\sum_{1\leqslant i\leqslant r} X \} \subseteq\sum_{1\leqslant i\leqslant r}\{H * X\}+ [H, K]^+.$$ 
 For $b \in H$ and $c \in C$, note that 
 \begin{eqnarray*} 
 b*(-c) &=& \lambda_b(-c) + c = -(-c + \lambda_b(c) -c +c) = -(-c + (b*c) + c)\\
 &=& \big(-c - (b*c) + c + (b*c)\big) - (b*c).
 \end{eqnarray*}
 Hence, since $\{H*C\}$ and $[K, K]^+$ are of $m$-bounded size, $\{H*(-C)\}$ is $m$-bounded. Thus,
 $$\{H*X\} \subseteq \{H*C\} + \{H*(-C)\} + [H, K]^+.$$
 Hence, it follows that  $|\{H*D\}|$ is $m$-bounded.

Let $\{1=a_1,a_2,\dots,a_t\}$ be a transversal of $D$ in $(K,+)$. We now get 
$$\{H*K\} \subseteq \{H*D\}+ \sum_{1\leqslant i\leqslant t}(H*a_i) + [H, K]^+,$$ 
which has $m$-bounded size.

For the first assertion, it remains to pass from the mixed set to the additive subgroup it generates.
Since the cardinality of $\{H * K\}$ is $m$-bounded, it suffices to prove that each element of $\{H * K\}$ is of $m$-bounded order modulo the subgroup $\mathcal{H}:=\gen{[H,K]^+, [H, H]^+, [K, K]^+}^+$. If $s =a * b$ for $a \in H$ and $b \in K$, then, modulo  $\mathcal{H}$, we get
\[
 n  s = a*(n  b)
\]
for every positive integer $n$. Let $t=|(K,+):D|$. The action of the cyclic subgroup generated by $b$ on the left cosets of $D$ shows that there exists an integer $1\leqslant n\leqslant t$ such that $nb\in D$. Let $R=|\{H*D\}|$, which is $m$-bounded. For $j=0,\ldots,R$, we have, modulo $\mathcal H$,
$$
jns=a*(jnb),
$$
and $jnb\in D$. Hence the $R+1$ elements $jns+\mathcal H$ take at most $R$ values. Thus two of them coincide, and so $kns\in\mathcal H$ for some $1\leqslant k\leqslant R$. Therefore the order of $s+\mathcal H$ is $m$-bounded in the additive abelian group $HK/\mathcal H$. Hence, the subgroup $\overline{H * K}$ generated by the set $\{s + \mathcal{H} \mid s \in \{H * K\}\}$ is of $m$-bounded order. Since $|\mathcal{H}|$ is $m$-bounded, the inverse image of $\overline{H * K}$ in $HK$ is of $m$-bounded order. This proves that $H*K$ is of $m$-bounded order. 

Finally, assume that $K=B$. Then, as proved above, $|H*B|$ and $|B*H|$ are $m$-bounded. Also, as noted above,  $|[H, B]^+|$ is $m$-bounded.  Since $[H,B]^b/[H,B]^+$ is abelian and additively generated by  the sets $\{H * B\}$ and $\{B*H\}$, it follows that it is of $m$-bounded order. Thus  $|[H,B]^b|$ is  $m$-bounded, and the proof is complete.
\end{proof}

As a consequence of Theorem \ref{lem:relative-BFC}, taking $H = K = B$, we get a genuine BFC-type statement for the second term of the lower central series of the brace $B$.

\begin{cor}\label{cor:bfc-gamma}
Let $B$ be a finite skew brace such that  for all $a \in B$ the brace centralizer $\Cb_B(a)$ has index at most $m$ in $(B,\circ)$. Then $|\Gamma_2(B)|$ is $m$-bounded.
\end{cor}

We now turn our attention to the second theme and prove the following result, which plays a key role in the proof of Theorem \ref{main-2}.  
\begin{thm}\label{maintheorem}
    Let $(B,+,\circ)$ be a finite skew brace from the class $\mathcal{T}$ with $\Pb(B)\geqslant \epsilon > 0$. Then there exist sub-skew braces $B_2 \leqslant B_1$ of  $B$ such that $B_2$ is an ideal of $B_1$, $B_1/B_2$ is a trivial brace, and $|B_2|$ and $|B:B_1|$ are $\epsilon$-bounded.
\end{thm}
\begin{proof}
Since $B\in\mathcal T$, Proposition~\ref{TS-S-L} shows that $\Cb_B(x)$ is a sub-skew brace of $B$ and
$\Cb_B(x)=\Cb_B(-x)=\Cb_B(x^{-1})$ for every $x\in B$. Put
$$
X=\{x\in B:[B:\Cb_B(x)]\leqslant 2/\epsilon\}.
$$
Then $X$ is symmetric with respect to both operations and contains the identity. Moreover,
\begin{align*}
\epsilon|B|^2\leqslant \Pb(B)|B|^2
&=\sum_{x\in X}|\Cb_B(x)|+\sum_{x\in B-X}|\Cb_B(x)|\\
&\leqslant |X||B|+\frac{\epsilon}{2}(|B|-|X|)|B|,
\end{align*}
so $|X|\geqslant \frac{\epsilon}{2}|B|$.

We shall use the following observation repeatedly. If $u,v\in B$, then
$$
\Cb_B(u)\cap\Cb_B(v)\leqslant \Cb_B(u+v)\cap\Cb_B(u\circ v).
$$
Indeed, if $c\in\Cb_B(u)\cap\Cb_B(v)$, then, by symmetry of brace-commuting, $u,v\in\Cb_B(c)$. Since $\Cb_B(c)$ is a sub-skew brace, it contains both $u+v$ and $u\circ v$, and the claim follows. Consequently, if $x$ is obtained from elements $x_1,\ldots,x_l$ by using only one of the two group operations and $[B:\Cb_B(x_i)]\leqslant m$ for every $i$, then
$$
[B:\Cb_B(x)]\leqslant m^l.
$$

Set $r=\lceil2/\epsilon\rceil$. By Lemma~\ref{Eb-lem}, applied to the additive group,
$$
D_1=\langle X\rangle^+=X^{3r}.
$$
Thus $|(B,+):D_1|\leqslant 2/\epsilon$, and
$[B:\Cb_B(x)]\leqslant(2/\epsilon)^{3r}$ for every $x\in D_1$.

We now alternate the two closures. If $D_i$ is an additive subgroup, put
$X_{i+1}=D_i\cup D_i^{-1}$ and $D_{i+1}=\langle X_{i+1}\rangle^\circ$; if $D_i$ is a multiplicative subgroup, put
$X_{i+1}=D_i\cup(-D_i)$ and $D_{i+1}=\langle X_{i+1}\rangle^+$. In either case $X_{i+1}$ is symmetric for the operation used to define $D_{i+1}$, contains the identity, and contains $D_i$. Moreover $D_i$ contains $X$, so
$|X_{i+1}|\geqslant|X|\geqslant\frac{\epsilon}{2}|B|$. Hence Lemma~\ref{Eb-lem}, with the same $r$, gives
$$
D_{i+1}=X_{i+1}^{3r},
\qquad
|B:D_{i+1}|\leqslant 2/\epsilon,
$$
where the index is taken in the corresponding underlying group. The preceding observation, together with
$\Cb_B(x)=\Cb_B(-x)=\Cb_B(x^{-1})$, shows inductively that the centralizer indices of all elements of $D_i$ are bounded by a function of $\epsilon$ and $i$ only.

The sets $D_i$ form an ascending chain. Each $D_i$ is a subgroup of one of the two underlying groups, and therefore $|D_i|$ divides $|B|$. Since $|B:D_i|\leqslant2/\epsilon$, every strict inclusion $D_i<D_{i+1}$ strictly decreases the positive integer $|B|/|D_i|$, which is at most $2/\epsilon$. Thus the chain stabilizes after an $\epsilon$-bounded number of steps. Let $D$ be a stable term. By construction, $D$ is closed under both group operations, and hence it is a sub-skew brace of $B$. Moreover $|B:D|$ is $\epsilon$-bounded and there is an $\epsilon$-bounded integer $m$ such that
$$
[D:\Cb_D(x)]\leqslant[B:\Cb_B(x)]\leqslant m
$$
for every $x\in D$.

Applying Corollary~\ref{cor:bfc-gamma} to $D$, we obtain that $|\Gamma_2(D)|$ is $\epsilon$-bounded. We therefore take
$$
B_1=D\qquad\text{and}\qquad B_2=\Gamma_2(D).
$$
By definition, $B_2$ is an ideal of $B_1$, while $B_1/B_2$ is a trivial brace. This proves the theorem.
\end{proof}

We also record a converse.

\begin{thm}
     Let $B_1 \leqslant B_2$ be sub-skew braces of a finite skew brace $B$ such that $B_1$ is an ideal of $B_2$ and $|B_1|\leqslant n_1$ and $[B:B_2]\leqslant n_2$ and $B_2/B_1$ is a trivial brace. Then $\Pb(B)\geqslant \epsilon(n_1,n_2)$.
\end{thm}
\begin{proof}
   Let $x\in B_2$ and $K_{B_2}(x)=\{b\circ x \circ b^{-1},b+ x - b,\lambda_b(x),x\circ b-b:b\in B_2\}$. We claim that $K_{B_2}(x)$ lies in a single coset of $B_1$ in $B_2$. We consider two cases; when $x \in B_1$ and otherwise. If $x \in B_1$, then the claim follows easily as $B_1$ is an ideal of $B_2$. So assume that $x \in B_2 \setminus B_1$. Then, since $B_2/B_1$ is a trivial brace, for any $c \in B_2$, we get
  $c \circ x \circ c^{-1} = x \mod{B_1}$, $c + x - c = x \mod{B_1}$, $\lambda_c(x) = x \mod{B_1}$ and $x \circ c - c = x \mod{B_1}$.
This proves the claim.
  
We prove the centralizer estimate used below. Set $L_x:=\Cb_{B_2}^l(x) = \{b \in B_2 \mid \lambda_b(x) = x\}$. The fibers of the map
\[
 b\longmapsto \lambda_b(x)
\]
are cosets of $L_x$ in $(B_2,\circ)$, and its image is contained in $K_{B_2}(x)$. Hence
\[
|B_2:L_x| \leqslant |K_{B_2}(x)|.
\]
Now consider on $L_x$ the map
\[
u\longmapsto\big(u+x-u,\,u\circ x\circ u^{-1}\big).
\]
We claim that each fiber is contained in a coset of $\Cb_{B_2}(x)$. Indeed, let $u,v\in L_x$ have the same image and put $d=v^{-1}\circ u\in L_x$. Since $u=v\circ d$ and $\lambda_v(x)=x$, the equality $u+x-u=v+x-v$ gives
\[
\lambda_v(d)+x-\lambda_v(d)=x,
\]
and applying $\lambda_v^{-1}$ yields $d+x-d=x$. The equality
$u\circ x\circ u^{-1}=v\circ x\circ v^{-1}$ gives
$d\circ x=x\circ d$. Since $d\in L_x$, we have $d\circ x=d+x=x+d$, and therefore $\lambda_x(d)=d$. Thus $d\in\Cb_{B_2}(x)$. Consequently
\[
|L_x:\Cb_{B_2}(x)|\leqslant |K_{B_2}(x)|^2,
\]
and hence
\[
|B_2:\Cb_{B_2}(x)| \leqslant |B_2:L_x| |L_x:\Cb_{B_2}(x)| \leqslant |K_{B_2}(x)|^3\leqslant n_1^3\leqslant n_1^4.
\]
Thus $|\Cb_{B_2}(x)|\geqslant \frac{|B_2|}{n_1^4}$. Now
   $$|B|^2\Pb(B) = \sum_{x\in B}|\Cb_{B}(x)| \geqslant \sum_{x\in B_2}|\Cb_{B_2}(x)|
      \geqslant \frac{|B_2||B_2|}{n_1^4} \geqslant \frac{|B|^2}{n_1^4n_2^2}.
   $$
   Hence, $\Pb(B)\geqslant \frac{1}{n_1^4n_2^2}$. The proof is complete by choosing $\epsilon(n_1,n_2)=\frac{1}{n_1^4n_2^2}$.
\end{proof}

As a nice side remark, we get
\begin{pro}
    Let $B$ be a skew brace such that $|\Gamma_2(B)|\leqslant m $. Then $|(B,\circ):\Cb_B(x)| \leqslant m^3$ for all $x\in B$.
\end{pro}
\begin{proof}
  As above, we define  $K_{B}(x) =\{b\circ x \circ b^{-1},b+ x - b,\lambda_b(x),x\circ b-b:b\in B\}$. Then, as proved in the preceding result,  $K_{B}(x)$ lies in a single coset of $\Gamma_2(B)$ in $B$. As $|(B,\circ):\Cb_B(x)| \leqslant |K_{B}(x)|^3\leqslant|\Gamma_2(B)|^3$, it follows that $|(B,\circ):\Cb_B(x)| \leqslant m^3$ for all $x\in B$. 
\end{proof}

One might wonder whether Theorem \ref{maintheorem} holds for all skew braces in full generality. We now exhibit a class of skew braces to show that this is not true.

\begin{thm}
There exists a family of finite skew braces $(B_n)_{n\geqslant 1}$ such that $\Pb(B_n)>1/27$ for every $n$, but there do not exist functions $f$ and $g$ such that, for every $n$, one can find sub-skew braces $B_2\leqslant B_1\leqslant B_n$ with $B_1/B_2$ trivial and
$|B_2|\leqslant f(1/27)$ and $[B_n:B_1]\leqslant g(1/27)$.
\end{thm}
\begin{proof}
Let $V_n=\mathbb F_3^{2n}$ and let $\beta:V_n\times V_n\to\mathbb F_3$ be a non-degenerate alternating bilinear form. Let $U_m=\mathbb F_7^m$ and set $B_{n,m}=U_m\times V_n\times\mathbb F_3$. Since the element $2$ has order $3$ in $\mathbb F_7^\times$, define
$$
(u,v,z)+(u',w,t)=(u+u',v+w,z+t+\beta(v,w))$$
and 
$$
(u,v,z)\circ(u',w,t)=(u+2^zu',v+w,z+t).
$$
The first operation is associative by bilinearity of $\beta$, since
$\beta(v,w)+\beta(v+w,r)=\beta(w,r)+\beta(v,w+r)$, while the second one is the semidirect product associated with the action of $\mathbb F_3$ on $U_m$ given by multiplication by $2^z$. Since $\beta$ is alternating, the additive inverse of $(u,v,z)$ is $(-u,-v,-z)$. Hence, for $x=(u,v,z)$ and $y=(u',w,t)$,
$$
\begin{aligned}
\lambda_x(y)
&=-x+(x\circ y)
 =(2^zu',w,t-\beta(v,w)),\\
x*y
&=\lambda_x(y)-y
 =((2^z-1)u',0,-\beta(v,w)),\\
[x,y]^+
&=(0,0,\beta(v,w)-\beta(w,v))
 =(0,0,2\beta(v,w)).
\end{aligned}
$$
Moreover $\lambda_x\lambda_y=\lambda_{x\circ y}$, because
$2^{z+t}=2^z2^t$ and $\beta(v+w,r)=\beta(v,r)+\beta(w,r)$. Thus $B_{n,m}$ is a skew brace.

We now compute $\Pb(B_{n,m})$. The elements $x=(u,v,z)$ and $y=(u',w,t)$ commute if and only if
$\beta(v,w)=0$, $(2^z-1)u'=0$ and $(2^t-1)u=0$. If $v=0$, every $w\in V_n$ is orthogonal to $v$, giving $3^{2n}$ choices. If $v\neq0$, non-degeneracy of $\beta$ implies that $v^\perp$ has codimension one, hence there are exactly $3^{2n-1}$ choices for $w$. Therefore the number of pairs $(v,w)$ satisfying $\beta(v,w)=0$ is
$3^{2n}+(3^{2n}-1)3^{2n-1}=3^{4n-1}+2\cdot3^{2n-1}$.

Put $Q=|U_m|=7^m$. If $z=0$, then $(2^z-1)u'=0$ imposes no condition on $u'$, so there are $Q$ choices. If $z=1$ or $z=2$, then $2^z-1\neq0$ in $\mathbb F_7$, so necessarily $u'=0$. Thus there are $Q+2$ possible pairs $(z,u')$, and similarly $Q+2$ possible pairs $(t,u)$. Since $|B_{n,m}|=Q3^{2n+1}$, it follows that
$$
\Pb(B_{n,m})
=\frac{(3^{4n-1}+2\cdot3^{2n-1})(7^m+2)^2}
{7^{2m}3^{4n+2}}
=\left(\frac1{27}+\frac{2}{3^{2n+3}}\right)
\left(1+\frac{2}{7^m}\right)^2
>\frac1{27}.
$$
For instance, $|B_{1,1}|=189$ and $\Pb(B_{1,1})=11/147$.

Let now $\pi:B_{n,m}\to C_n=V_n\times\mathbb F_3$ be the natural projection. The induced operations on $C_n$ are
$(v,z)+(w,t)=(v+w,z+t+\beta(v,w))$ and
$(v,z)\circ(w,t)=(v+w,z+t)$. Put $D_n=V_n\times\{0\}$. Suppose that $J\leqslant C_n$ is a sub-skew brace contained in $D_n$. Since the multiplicative operation on $D_n$ is ordinary vector addition, $J=L\times\{0\}$ for some subspace $L\leqslant V_n$. Additive closure gives
$(v,0)+(w,0)=(v+w,\beta(v,w))\in J$ for all $v,w\in L$, and therefore $\beta(L,L)=0$. Thus $L$ is totally isotropic. Since $\beta$ is non-degenerate and $\dim V_n=2n$, we have $\dim L\leqslant n$, so $|J|\leqslant3^n$ and consequently $[C_n:J]\geqslant3^{n+1}$.

Let $B_2\leqslant B_1\leqslant B_{n,m}$ and assume that $B_1/B_2$ is trivial. Set $J=\pi(B_1)$ and $W=B_1\cap U_m$. Since $\ker(\pi|_{B_1})=W$, we have $|B_1|=|W||J|$, hence
$|B_{n,m}:B_1|=|U_m:W||C_n:J|$. If $J\subseteq D_n$, then $[B_{n,m}:B_1]\geqslant3^{n+1}$. Assume instead that $J\not\subseteq D_n$. Then $B_1$ contains some $x=(u,v,z)$ with $z\neq0$. For every $w\in W$, since $(w,0,0)\in B_1$ and $B_1/B_2$ is trivial, we have
$x*(w,0,0)=((2^z-1)w,0,0)\in B_2$. Since $z\in\{1,2\}$ and $2^z-1\neq0$ in $\mathbb F_7$, multiplication by $2^z-1$ is an automorphism of $W$, and therefore $W\leqslant B_2$. Thus
$$
|B_{n,m}:B_1|\geqslant3^{n+1}
\qquad\text{or}\qquad
|B_2|\,|B_{n,m}:B_1|
\geqslant |W|\,|U_m:W|
=7^m.
$$

Finally, set $m=n$ and write $B_n=B_{n,n}$. Then $\Pb(B_n)>1/27$ for every $n$. Suppose that there were constants $F$ and $G$, depending only on $1/27$, such that for every $n$ one could find $B_2\leqslant B_1\leqslant B_n$ with $B_1/B_2$ trivial, $|B_2|\leqslant F$ and $|B_n:B_1|\leqslant G$. Choose $n$ such that $3^{n+1}>G$ and $7^n>FG$. The first alternative above is then impossible, while the second yields
$|B_2|[B_n:B_1]\geqslant7^n>FG$, a contradiction.
\end{proof}

In the next three subsections of this section we improve Theorem \ref{maintheorem} by proving the following result whose proof follows from   Theorem~\ref{thm:bfc-ideals-two-sided} (two-sided case),  Theorem~\ref{thm:bfc-ideals-symmetric} (symmetric case) and  Theorem~\ref{thm:bfc-ideals-lambda-homomorphic} ($\lambda$-homomorphic case). A group-theoretic analogue of this result was proved by P. M. Neumann \cite{PMN89}, and similar results for finite rings were obtained in  \cite{SV24}. 

\begin{thm}\label{thm:bfc-ideals-class-T}
Let $B$ be a finite skew brace belonging to the class $\mathcal T$ and suppose that $\Pb(B)\geqslant\epsilon>0$. Then there exist ideals $B_2 \leqslant B_1\trianglelefteq B$ such that $B_1/B_2$ is a trivial brace, while $|B_2|$ and $|B:B_1|$ are bounded in terms of $\epsilon$ only.
\end{thm}

We remark that  the proof still uses Theorem \ref{maintheorem}.


\subsection{Two-sided Skew Braces}

In this subsection we deal with two-sided skew braces.
\begin{lemma}\label{lem:bounded-ideal-closure-two-sided}
Let $B$ be a finite two-sided skew brace, and let $D\leqslant B$ be a
sub-skew brace of index at most $n$. Put $N=\Gamma_2(D)$. If $|N|\leqslant m$,
then the ideal of $B$ generated by $N$ has $(m,n)$-bounded order.
\end{lemma}

\begin{proof}
Let $G_\lambda=(B,+)\rtimes\lambda(B)$ and let
$H=(D,+)\rtimes\lambda(D)$. Since $D$ is a sub-skew brace, $H\leqslant G_\lambda$,
and $|G_\lambda:H| \leqslant |B:D|^2\leqslant n^2$. Moreover $N=\Gamma_2(D)$ is an
ideal of $D$, hence it is normal in $(D,+)$ and invariant under every
$\lambda_d$, $d\in D$. Thus $N$ is normalized by $H$.

We use the following standard consequence of Dietzmann's lemma: if $G$ has
a subgroup $H$ of index at most $r$, and $N\unlhd H$ has order at most $m$,
then the normal closure $N^G$ has $(m,r)$-bounded order. Applying this in
$G_\lambda$, the smallest additive normal subgroup $S$ of $(B,+)$ containing
$N$ and invariant under all $\lambda_b$, $b\in B$, has $(m,n)$-bounded order.
Thus $S$ is a strong left ideal of $B$.

Since $B$ is two-sided, every multiplicative conjugation
$x\mapsto b\circ x\circ b^{-1}$ is an automorphism of the skew brace $B$.
For $d\in D$ we have $d\circ N\circ d^{-1}=N$, because $N$ is an ideal of
$D$. Hence $d\circ S\circ d^{-1}=S$ for every $d\in D$. Therefore the family
of multiplicative conjugates of $S$ in $B$ has cardinality at most
$|B:D| \leqslant n$. Let these conjugates be $S_1,\ldots,S_t$, with $t\leqslant n$, and
put $K=S_1+\cdots+S_t$ inside the additive group $(B,+)$. Then $|K|$ is
$(m,n)$-bounded. Moreover $K$ is normal in $(B,+)$, invariant under all
$\lambda_b$, and normal in $(B,\circ)$, because multiplicative conjugation
permutes the subgroups $S_i$. Thus $K$ is an ideal of $B$. Since
$N\leqslant K$, the ideal of $B$ generated by $N$ is contained in $K$, and hence
has $(m,n)$-bounded order.
\end{proof}

\begin{lem}\label{lem:finite-index-ideal-core-two-sided}
Let $B$ be a finite two-sided skew brace, and let $A$ be a sub-skew brace of $B$.
If $|B:A|=m$, then there exists an ideal $I$ of $B$ such that $I\leqslant A$ and
$|B:I|$ is bounded in terms of $m$ only.
\end{lem}

\begin{proof}
Let $T$ be a set of representatives for the left cosets of $(A,\circ)$ in
$(B,\circ)$, chosen so that $1\in T$. Thus $|T|=m$. Set
$L_1=\bigcap_{t\in T}\lambda_t(A)$. Since each $\lambda_t$ is an automorphism
of the additive group $(B,+)$, each $\lambda_t(A)$ is an additive subgroup of
index $m$. Hence $L_1$ has additive index bounded in terms of $m$ only.
Moreover, since $1\in T$, we have $L_1\leqslant A$.

We claim that $L_1$ is invariant under every $\lambda_b$, with $b\in B$. Indeed,
for fixed $b\in B$ and $t\in T$, there exist $t'\in T$ and $a\in A$ such that
$b\circ t=t'\circ a$. Since $A$ is a sub-skew brace, $\lambda_a(A)=A$. Therefore
$\lambda_{b\circ t}(A)=\lambda_{t'\circ a}(A)=\lambda_{t'}\lambda_a(A)
=\lambda_{t'}(A)$. As $t$ varies in $T$, the element $t'$ also varies in $T$.
Thus $\lambda_b(L_1)=L_1$, and so $L_1$ is $\lambda$-invariant.

Now let $S$ be a set of representatives for the additive cosets of $L_1$ in
$(B,+)$, and set $L=\bigcap_{s\in S}(s+L_1-s)$. Then $L$ is a normal subgroup
of $(B,+)$ contained in $L_1$, hence contained in $A$. Its additive index is
bounded in terms of the additive index of $L_1$, and therefore in terms of $m$
only. Moreover, since $L_1$ is invariant under all maps $\lambda_b$ and each
$\lambda_b$ is an automorphism of $(B,+)$, the map $\lambda_b$ permutes the
additive cosets of $L_1$. Hence $\lambda_b(L)=L$ for every $b\in B$. Therefore
$L$ is a strong left ideal of $B$, contained in $A$, and its index is bounded
in terms of $m$ only.

It remains to make the subgroup normal with respect to the multiplicative
group. Since $L$ is a left ideal, its additive and multiplicative cosets
coincide; in particular, the multiplicative index $|B:L|^{\circ}$ is bounded
in terms of $m$ only. Let $U$ be a set of representatives for the left cosets
of $(L,\circ)$ in $(B,\circ)$, chosen so that $1\in U$, and define
$I=\bigcap_{u\in U}u\circ L\circ u^{-1}$.

Since $B$ is two-sided, every multiplicative conjugation $x\mapsto u\circ x
\circ u^{-1}$ is an automorphism of the skew brace $B$. Hence each
$u\circ L\circ u^{-1}$ is again a strong left ideal of $B$, and so their
intersection $I$ is a strong left ideal of $B$. By construction, $I$ is the
core of $L$ in the multiplicative group $(B,\circ)$, hence $(I,\circ)$ is
normal in $(B,\circ)$. Therefore $I$ is an ideal of $B$. Since $1\in U$, we
have $I\leqslant L\leqslant A$. Finally, $|B:I|$ is bounded in terms of $|B:L|$ only,
and hence in terms of $m$ only.
\end{proof}

\begin{thm}\label{thm:bfc-ideals-two-sided}
Let $B$ be a finite two-sided skew brace. If $\Pb(B)\geqslant\epsilon>0$, then there exist ideals $B_2\leqslant B_1\trianglelefteq B$ such that $B_1/B_2$ is a trivial brace, while $|B:B_1|$ and $|B_2|$ are bounded in terms of $\epsilon$ only.
\end{thm}
\begin{proof}
By Theorem~\ref{maintheorem}, there exists a sub-skew brace
$D\leqslant B$ such that $|B:D|$ and $|\Gamma_2(D)|$ are $\epsilon$-bounded. Put $N=\Gamma_2(D)$. By Lemma~\ref{lem:bounded-ideal-closure-two-sided}, the ideal $K$ of $B$ generated by $N$ has $\epsilon$-bounded order.

Pass to $\overline B=B/K$. Since $N\leqslant K$, the image $\overline D$ of $D$ satisfies $\Gamma_2(\overline D)= 1$, and hence $\overline D$ is a trivial brace. Moreover $|\overline B:\overline D| \leqslant |B:D|$, so this index is $\epsilon$-bounded. By Lemma~\ref{lem:finite-index-ideal-core-two-sided}, applied to the
finite-index sub-skew brace $\overline D$ of $\overline B$, there exists an
ideal $\overline B_1$ of $\overline B$ such that $\overline B_1\leqslant\overline D$
and $|\overline B:\overline B_1|$ is bounded in terms of
$|\overline B:\overline D|$ only. Since $\overline B_1\leqslant \overline D$, the ideal $\overline B_1$ is a trivial brace.

Let $B_1$ be the full preimage of $\overline B_1$ in $B$, and put $B_2=K$. Then $B_1$ and $B_2$ are ideals of $B$, $B_2\leqslant B_1$, and
$B_1/B_2\simeq\overline B_1$ is trivial. Finally $|B_2|=|K|$ and
$|B:B_1| = |\overline B:\overline B_1|$ are $\epsilon$-bounded.
\end{proof}

\subsection{Symmetric Skew Braces}

We first record a permutation-theoretic observation for symmetric skew braces. Throughout this subsection, permutations are composed from right to left.

\begin{lem}\label{lem:symmetric-permutation-bridge}
Let $B$ be a finite symmetric skew brace. For $a\in B$, define $\tau_a(x)=a+x$ and $\sigma_a(x)=a\circ x$, and set $A=\{\tau_a:a\in B\}$ and $M=\{\sigma_a:a\in B\}$. Then $A$ and $M$ normalize each other, and hence $G:=AM=\langle A,M\rangle$ is a subgroup of $\operatorname{Sym}(B)$. Moreover, if $R\unlhd G$, then
$$
K_R:=\{x\in B:\tau_x\in R\text{ and }\sigma_x\in R\}
$$
is an ideal of $B$. If $|R|\leqslant s$, then $|K_R|\leqslant s$, while if $|G:R|\leqslant q$, then $|B:K_R|\leqslant q^2$.
\end{lem}

\begin{proof}
Let $\lambda$ be the lambda map of $(B,+,\circ)$ and let $\mu$ be the lambda map of the reversed skew brace $(B,\circ,+)$. Thus $\mu_a(x)=a^{-1}\circ(a+x)$. Since $a+x=a\circ\mu_a(x)=a+\lambda_a(\mu_a(x))$, we have $\mu_a=\lambda_a^{-1}$. As $B$ is symmetric, $\mu_a\in\Aut(B,\circ)$, so $\lambda_a$ is an automorphism of both $(B,+)$ and $(B,\circ)$. Since $\mu:(B,+)\to\Aut(B,\circ)$ is a homomorphism, $\lambda_{a+b}=\lambda_b\lambda_a$, while, as usual, $\lambda_{a\circ b}=\lambda_a\lambda_b$.

Since $\sigma_a=\tau_a\lambda_a$ and $\lambda_a=\tau_a^{-1}\sigma_a$, we obtain
$$
\sigma_b\tau_x\sigma_b^{-1}=\tau_{\,b+\lambda_b(x)-b},\qquad
\tau_b\sigma_x\tau_b^{-1}=\sigma_{\,b\circ\lambda_b^{-1}(x)\circ b^{-1}}.
$$
Hence $A$ and $M$ normalize each other.

Let now $R\unlhd G$ and put $Q=G/R$. The maps $\Phi_+:(B,+)\to Q\times Q^{\mathrm{op}}$, $x\mapsto(\tau_xR,\lambda_xR)$, and $\Phi_\circ:(B,\circ)\to Q\times Q$, $x\mapsto(\sigma_xR,\lambda_xR)$, are homomorphisms, and both have kernel $K_R$, since $\sigma_x=\tau_x\lambda_x$ and $\tau_x=\sigma_x\lambda_x^{-1}$. Thus $K_R$ is normal in both $(B,+)$ and $(B,\circ)$.

Finally, $\lambda_b=\tau_b^{-1}\sigma_b\in G$. Hence, for $x\in K_R$, normality of $R$ gives $\tau_{\lambda_b(x)}=\lambda_b\tau_x\lambda_b^{-1}\in R$ and $\sigma_{\lambda_b(x)}=\lambda_b\sigma_x\lambda_b^{-1}\in R$, so $\lambda_b(x)\in K_R$. Therefore, $K_R$ is an ideal of $B$. If $|R|\leqslant s$, injectivity of $x\mapsto\tau_x$ gives $|K_R|\leqslant s$. If $|G:R|\leqslant q$, then $|Q|\leqslant q$ and $|B:K_R|=|\operatorname{Im}\Phi_+|\leqslant |Q|^2\leqslant q^2$.
\end{proof}

\begin{lem}\label{lem:bounded-ideal-closure-symmetric}
Let $B$ be a finite symmetric skew brace, let $D\leqslant B$ be a sub-skew brace with $|B:D|\leqslant n$, and let $N$ be an ideal of $D$ with $|N|\leqslant m$. Then the ideal of $B$ generated by $N$ has $(m,n)$-bounded order.
\end{lem}

\begin{proof}
For $X\subseteq B$, write $A_X=\{\tau_x:x\in X\}$ and $M_X=\{\sigma_x:x\in X\}$. By Lemma~\ref{lem:symmetric-permutation-bridge}, $G=A_BM_B$ is a group, and $H=A_DM_D\leqslant G$. Moreover,
$$
|G:H|=\frac{|B|^2}{|D|^2}\frac{|A_D\cap M_D|}{|A_B\cap M_B|}\leqslant |B:D|^2\leqslant n^2.
$$

Set $E=A_NM_N$. Since $N\unlhd(D,+)$, $N\unlhd(D,\circ)$ and $\lambda_d(N)=N$ for every $d\in D$, the formulas
$$
\sigma_d\tau_x\sigma_d^{-1}=\tau_{\,d+\lambda_d(x)-d},\qquad
\tau_d\sigma_x\tau_d^{-1}=\sigma_{\,d\circ\lambda_d^{-1}(x)\circ d^{-1}}
$$
show that $E\unlhd H$. Also $|E|\leqslant |N|^2\leqslant m^2$. Hence, by the standard consequence of Dietzmann's lemma, the normal closure $R:=E^G$ has $(m,n)$-bounded order.

By Lemma~\ref{lem:symmetric-permutation-bridge}, $K_R$ is an ideal of $B$ and $|K_R|\leqslant |R|$. Since $\tau_x,\sigma_x\in E\leqslant R$ for every $x\in N$, we have $N\leqslant K_R$. Therefore the ideal of $B$ generated by $N$ is contained in $K_R$ and has $(m,n)$-bounded order.
\end{proof}

\begin{lem}\label{lem:finite-index-ideal-core-symmetric}
Let $B$ be a finite symmetric skew brace and let $D\leqslant B$ be a sub-skew brace with $|B:D|\leqslant n$. Then there exists an ideal $I$ of $B$ such that $I\leqslant D$ and
$$
|B:I|\leqslant \big((n^2)!\big)^2.
$$
\end{lem}

\begin{proof}
Put $G=A_BM_B$ and $H=A_DM_D$. As above, $|G:H|\leqslant n^2$. Let $C=\operatorname{core}_G(H)$. Then $C\unlhd G$, $C\leqslant H$, and $|G:C|\leqslant (n^2)!$. Set $I=K_C$. By Lemma~\ref{lem:symmetric-permutation-bridge}, $I$ is an ideal of $B$ and $|B:I|\leqslant |G:C|^2\leqslant((n^2)!)^2$.

It remains to show that $I\leqslant D$. If $x\in I$, then $\tau_x\in C\leqslant H=A_DM_D$, so $\tau_x=\tau_d\sigma_e$ for some $d,e\in D$. Evaluating at the common identity element gives $x=d+e\in D$. Hence $I\leqslant D$.
\end{proof}

\begin{thm}\label{thm:bfc-ideals-symmetric}
Let $B$ be a finite symmetric skew brace with $\Pb(B)\geqslant\epsilon>0$. Then there exist ideals $B_2\leqslant B_1\trianglelefteq B$ such that $B_1/B_2$ is a trivial brace, while $|B_2|$ and $|B:B_1|$ are bounded in terms of $\epsilon$ only.
\end{thm}

\begin{proof}
By Theorem~\ref{maintheorem}, there exists a sub-skew brace $D\leqslant B$ such that $|B:D|$ and $|\Gamma_2(D)|$ are $\epsilon$-bounded and $D/\Gamma_2(D)$ is trivial. Put $N=\Gamma_2(D)$. By Lemma~\ref{lem:bounded-ideal-closure-symmetric}, the ideal $K$ of $B$ generated by $N$ has $\epsilon$-bounded order.

Pass to $\overline B=B/K$ and let $\overline D$ be the image of $D$. Then $\overline B$ is symmetric, $|\overline B:\overline D|\leqslant |B:D|$, and $\overline D$ is trivial because $N\leqslant K$. By Lemma~\ref{lem:finite-index-ideal-core-symmetric}, there exists an ideal $\overline B_1\trianglelefteq\overline B$ with $\overline B_1\leqslant\overline D$ and $|\overline B:\overline B_1|$ $\epsilon$-bounded. Hence $\overline B_1$ is trivial.

Let $B_1$ be the full preimage of $\overline B_1$ in $B$ and set $B_2=K$. Then $B_1$ and $B_2$ are ideals of $B$, $B_2\leqslant B_1$, $B_1/B_2\simeq\overline B_1$ is trivial, and
$$
|B_2|=|K|,\qquad |B:B_1|=|\overline B:\overline B_1|
$$
are $\epsilon$-bounded.
\end{proof}

\subsection{$\lambda$-homomorphic Skew Braces}

We now establish analogous results for finite $\lambda$-homomorphic skew braces.

\begin{lem}\label{lem:lambda-homomorphic-bridge}
Let $B$ be a finite $\lambda$-homomorphic skew brace,  $\Lambda=\lambda(B)$ and set $G:=(B,+)\rtimes\Lambda$, with multiplication $(x,\alpha)(y,\beta)=(x+\alpha(y),\alpha\beta)$. For $x\in B$, write $t_x=(x,\operatorname{id})$, $\ell_x=(1,\lambda_x)$ and $m_x=(x,\lambda_x)=t_x\ell_x$. If $R\unlhd G$, then
$$
I_R:=\{x\in B:t_x\in R\text{ and }\ell_x\in R\}
$$
is an ideal of $B$. Moreover, if $|R|\leqslant s$, then $|I_R|\leqslant s$, while if $|G:R|\leqslant q$, then $|B:I_R|\leqslant q^2$.
\end{lem}

\begin{proof}
Since $B$ is $\lambda$-homomorphic, $\lambda_{x+y}=\lambda_x\lambda_y$, and using the fact  $\lambda_{x\circ y}=\lambda_x\lambda_y$, we see that $\lambda_{\lambda_a(x)}=\lambda_x$  for every $a, x\in B$.
Let $Q:=G/R$. The map $\Phi_+:(B,+)\to Q\times Q$, $x\mapsto(t_xR,\ell_xR)$, is a homomorphism with kernel $I_R$. Moreover, since
$$
m_xm_y=(x+\lambda_x(y),\lambda_x\lambda_y)=(x\circ y,\lambda_{x\circ y})=m_{x\circ y},
$$
the map $\Phi_\circ:(B,\circ)\to Q\times Q$, $x\mapsto(m_xR,\ell_xR)$, is also a homomorphism with kernel $I_R$. Hence $I_R$ is normal in both $(B,+)$ and $(B,\circ)$.

If $x\in I_R$ and $a\in B$, then $t_{\lambda_a(x)}=\ell_at_x\ell_a^{-1}\in R$, while $\ell_{\lambda_a(x)}=\ell_x\in R$. Thus $\lambda_a(x)\in I_R$, and $I_R$ is an ideal of $B$. Finally, $|I_R|\leqslant |R|$ by injectivity of $x\mapsto t_x$, whereas if $|G:R|\leqslant q$, then $|B:I_R|=|\operatorname{Im}\Phi_+|\leqslant |Q|^2\leqslant q^2$.
\end{proof}

\begin{lem}\label{lem:bounded-ideal-closure-lambda-homomorphic}
Let $B$ be a finite $\lambda$-homomorphic skew brace,  $D\leqslant B$ be a sub-skew brace with $|B:D|\leqslant n$, and $N$ be an ideal of $D$ with $|N|\leqslant m$. Then the ideal of $B$ generated by $N$ has $(m,n)$-bounded order.
\end{lem}

\begin{proof}
Set $G :=(B,+)\rtimes\lambda(B)$ and $H:=(D,+)\rtimes\lambda(D)$. Since $\lambda:(B,+)\to\lambda(B)$ is a homomorphism, $|\lambda(B):\lambda(D)|\leqslant |B:D|$, and hence
$$
|G:H|=|B:D|\,|\lambda(B):\lambda(D)|\leqslant n^2.
$$

Set $E:=(N,+)\rtimes\lambda(N)$. We claim that $E$ is a normal subgroup of $H$. Since $N$ is an ideal of $D$, it is normal in $(D,+)$ and $\lambda_d(N)=N$ for every $d\in D$. Thus $t_dt_nt_d^{-1}=t_{d+n-d}\in E$, $\ell_dt_n\ell_d^{-1}=t_{\lambda_d(n)}\in E$, and $\ell_d\ell_n\ell_d^{-1}=(1,\lambda_{d+n-d})\in E$. Finally,
$$
t_d\ell_nt_d^{-1}=(d-\lambda_n(d),\lambda_n)=t_{d-\lambda_n(d)}\ell_n\in E,
$$
since, $N$ being an ideal of $D$, $\lambda_n(d)+N=d+N$, which implies that  $d-\lambda_n(d)\in N$; hence, we get $E\unlhd H$.

Since $|E|=|N|\,|\lambda(N)|\leqslant m^2$, the normal closure $R:=E^G$ has $(m,n)$-bounded order by the standard consequence of Dietzmann's lemma. By Lemma~\ref{lem:lambda-homomorphic-bridge}, $I_R$ is an ideal of $B$ with $(m,n)$-bounded order. Since $t_x,\ell_x\in E\leqslant R$ for every $x\in N$, we have $N\leqslant I_R$. Hence the ideal of $B$ generated by $N$ has $(m,n)$-bounded order.
\end{proof}

\begin{lem}\label{lem:finite-index-ideal-core-lambda-homomorphic}
Let $B$ be a finite $\lambda$-homomorphic skew brace and let $D\leqslant B$ be a sub-skew brace with $|B:D|\leqslant n$. Then there exists an ideal $I$ of $B$ such that $I\leqslant D$ and
$$
|B:I|\leqslant\big((n^2)!\big)^2.
$$
\end{lem}

\begin{proof}
Set $G :=(B,+)\rtimes\lambda(B)$ and $H :=(D,+)\rtimes\lambda(D)$. As above, $|G:H|\leqslant n^2$. Let $C=\operatorname{core}_G(H)$. Then $C\unlhd G$, $C\leqslant H$, and $|G:C|\leqslant(n^2)!$. Define 
$$I :=\{x\in B:t_x\in C\text{ and }\ell_x\in C\}.$$
Then, by Lemma~\ref{lem:lambda-homomorphic-bridge}, $I$ is an ideal of $B$ and $|B:I|\leqslant |G:C|^2\leqslant((n^2)!)^2$. If $x\in I$, then $t_x=(x,\operatorname{id})\in C\leqslant H=(D,+)\rtimes\lambda(D)$, whence $x\in D$. Thus $I\leqslant D$.
\end{proof}

\begin{thm}\label{thm:bfc-ideals-lambda-homomorphic}
Let $B$ be a finite $\lambda$-homomorphic skew brace satisfying $\Pb(B)\geqslant\epsilon>0$. Then there exist ideals $B_2\leqslant B_1\trianglelefteq B$ such that $B_1/B_2$ is a trivial brace, while $|B_2|$ and $|B:B_1|$ are bounded in terms of $\epsilon$ only.
\end{thm}

\begin{proof}
By Theorem~\ref{maintheorem}, there exists a sub-skew brace $D\leqslant B$ such that $|B:D|$ and $|\Gamma_2(D)|$ are $\epsilon$-bounded and $D/\Gamma_2(D)$ is trivial. Put $N=\Gamma_2(D)$. By Lemma~\ref{lem:bounded-ideal-closure-lambda-homomorphic}, the ideal $K$ of $B$ generated by $N$ has $\epsilon$-bounded order.

Pass to $\overline B=B/K$ and let $\overline D$ be the image of $D$. Then $\overline B$ is again $\lambda$-homomorphic, $|\overline B:\overline D|\leqslant |B:D|$, and $\overline D$ is trivial. By Lemma~\ref{lem:finite-index-ideal-core-lambda-homomorphic}, there exists an ideal $\overline B_1\trianglelefteq\overline B$ such that $\overline B_1\leqslant\overline D$ and $|\overline B:\overline B_1|$ is $\epsilon$-bounded. Hence $\overline B_1$ is trivial.

Let $B_1$ be the full preimage of $\overline B_1$ in $B$ and set $B_2=K$. Then $B_1$ and $B_2$ are ideals of $B$, $B_2\leqslant B_1$, $B_1/B_2\simeq\overline B_1$ is trivial, and
$$
|B_2|=|K|,\qquad |B:B_1|=|\overline B:\overline B_1|
$$
are $\epsilon$-bounded. This completes the proof.
\end{proof}


\section{The Probability that $\langle a,b\rangle$ is Trivial}

In this section we study the probability that two randomly chosen elements in a skew brace generate a trivial sub-skew brace. For a given skew brace $B$, this probability, denoted by $\Pt(B)$, is defined by
$$
\Pt(B)=\frac{|\{(a,b)\in B^2:\langle a,b\rangle\text{ is a trivial sub-skew brace}\}|}{|B|^2},
$$
We also define the trivializing  probability of a skew brace $B$, denoted by $P_{\lambda}(B)$, as follows 
$$
P_{\lambda}(B)=\frac{|\{(a,b)\in B^2:a*b= 1\}|}{|B|^2}.
$$
Clearly $\Pt(B)\leqslant P_{\lambda}(B)$. The  trivializing probability will be useful because it admits a direct group-theoretic interpretation, while
the additional conditions involved in $\Pt(B)$ will later give a stronger
conclusion. More precisely, if $P_{\lambda}(B)=1$, then $B$ is a trivial skew brace, that is $\circ = +$.

For $a\in B$, put
$T(a)=\{b\in B:\langle a,b\rangle\text{ is a trivial sub-skew brace}\}$. Thus
$\Pt(B)|B|^2=\sum_{a\in B}|T(a)|$. Notice that $T(a)$ need not be a
subgroup of either underlying group. Therefore the centralizer arguments
used in the preceding sections cannot be applied directly to the sets
$T(a)$.

We first describe exactly the pairs counted by $\Pt(B)$.

\begin{lem}\label{lem:two-generated-trivial}
Let $a,b\in B$. Then $\langle a,b\rangle$ is a trivial sub-skew brace if and only if
$$
a*a=a*b=b*a=b*b= 1.
$$
\end{lem}

\begin{proof}
The necessity is clear. Conversely, the four equalities imply that
$\lambda_a$ and $\lambda_b$ fix both $a$ and $b$. Hence, they fix $L=\langle a,b\rangle^+$ pointwise. Since
$\lambda:(B,\circ)\to\Aut(B,+)$ is a homomorphism, every multiplicative
word in $a,b$ acts trivially on $L$. Moreover $a\circ b=a+b$, and the
multiplicative inverses of $a$ and $b$ coincide with their additive
inverses. It follows inductively that the additive and multiplicative
subgroups generated by $a,b$ coincide and that the two operations agree
on this subgroup.
\end{proof}

Put $S(B)=\{x\in B:x*x= 1\}$. Every pair counted by $\Pt(B)$ belongs to
$S(B)^2$. Consequently,
$$
\Pt(B)\geqslant\epsilon
\quad \mbox{ implies }\quad
|S(B)|\geqslant\sqrt{\epsilon}\,|B|.
$$
We shall see below that the exponent $1/2$ is optimal.

We now interpret $P_\lambda(B)$ inside the natural semidirect product of $(B,+)$ by $(B,\circ)$. Set
$A=(B,\circ)$, $M=(B,+)$ and $G=M\rtimes_\lambda A$. Conjugation by the
copy of $a\in A$ on $M$ is $\lambda_a$; hence $a$ and $x\in M$ commute
in $G$ precisely when $a*x=1$. Thus
$$
P_\lambda(B)= \Pr(A,M).
$$
Notice that neither $\Pt(B)$ nor $P_\lambda(B)$ can control the internal
structure of either underlying group: a trivial skew brace has
$\Pt(B)=P_\lambda(B)=1$, while its common underlying group is arbitrary.

The following group-theoretic lemma is the ingredient which converts a
uniform control of the two mutual actions into a bounded mixed
commutator. Its abelian case is based on the stronger BFC-theorem
recalled in \cite[Lemma 2.2, Theorem 2.4]{DS22}.

\begin{lem}\label{lem:mixed-BFC}
Let $G=M \rtimes_\lambda A$, where $M\unlhd G$ and $A\leqslant G$. Suppose that
$|a^M|\leqslant m$ for all $a\in A$ and $|x^A|\leqslant m$ for all $x\in M$ for some positive integer $m$. Then $|[M,A]|$ is $m$-bounded.
\end{lem}

\begin{proof}
Suppose first that $A$ is abelian. Then $a^G=a^M$ for every $a\in A$.
Put $L=\langle A^G\rangle$. By \cite[Theorem 2.4]{DS22}, $|\gamma_2(L)|$ is
$m$-bounded. Since $L/\gamma_2(L)$ is abelian,  it follows by
\cite[Lemma 2.2]{DS22} that  $[M,A]$ has $m$-bounded order. Hence
the same holds before taking the quotient.

Now let $A$ be arbitrary. For each $a\in A$, the abelian case applied
to $M\langle a\rangle$ shows that $|[M,\gen{a}]|$ is $m$-bounded. Fix
$x\in M$. By the given hypothesis $|\{[x, a] \mid a \in A\}|$ is  $m$-bounded, each $[x,a]$ has $m$-bounded order. Since $[x,a]=(a^x)a^{-1}$, we get $|[x, a]|^M$ is $m^2$-bounded, and hence $|[x, a]^G|$ is $m$-bounded.   
Now it follows by Dietzmann's
lemma that $\gen{[x, a]}^G$ is $m$-bounded for all $x \in M$ and $a \in A$. Note that $[x, A]^G$ is a subgroup of the product $\Pi_{i=1}^{t} \gen{[x,a_i]}^G$, where $t = |\{[x, a] \mid a \in A\}|$ is $m$-bounded. Set 
$N_x=\langle[x,A]^G\rangle$.

Choose $y\in M$ with $|y^A|$ maximal, say $|y^A|=s\leqslant m$, write
$y^A=\{y^{b_1},\ldots,y^{b_s}\}$, and set
$M_1=C_M(b_1,\ldots,b_s)$. Then $|M:M_1|\leqslant m^s$. If
$u\in M_1$, the $s$ elements $y^{b_i}u=(yu)^{b_i}$ are distinct and,
by maximality, form the whole $A$-orbit of $yu$. Thus, for every
$a\in A$, one has $(yu)^a=(yu)^{b_i}$ for some $i$, and hence
$[u,a]\in N_y$. Therefore $[M_1,A]\leqslant N_y$.

Modulo $N_y$ we may assume $M_1\leqslant C_M(A)$. Choose
$y_1,\ldots,y_t$ representing the cosets of $M_1$ in $M$; here $t$ is
$m$-bounded. Then
$$
[M,A]\leqslant N_yN_{y_1}\cdots N_{y_t},
$$
and the right-hand side has $m$-bounded order. This completes the proof.
\end{proof}

We first obtain the conclusion which follows from the weaker assumption
$P_\lambda(B)\geqslant\epsilon$. The quantitative bound below is obtained
by using Lemma \ref{Eb-lem} (Eberhard's lemma \cite[Lemma 2.1]{SE15}; see also
\cite[Lemma 2.5]{DS22}).

\begin{thm}\label{thm:Plambda-optimal-form}
Let $B$ be a finite skew brace with $P_\lambda(B)\geqslant\epsilon>0$,
and put $r=\lfloor1/\epsilon\rfloor$. There exist
$A_0\unlhd(B,\circ)$ and strong left ideals $N\leqslant M_0\leqslant B$
such that
$$
|(B,\circ):A_0|\leqslant r,\quad
|(B,+):M_0|\leqslant r!,\quad
|N|\text{ is $\epsilon$-bounded},\quad
A_0*M_0\leqslant N.
$$
In particular, the image of the induced action
$(B,\circ)\to\Aut(M_0/N)$ has order at most $r$.
\end{thm}

\begin{proof}
Since $1/(r+1)<\epsilon$, choose
$0<\delta<((r+1)\epsilon-1)/r$ and define
$$
X=\{a\in A:|C_M(a)|\geqslant\delta|M|\},
\qquad
Y=\{x\in M:|C_A(x)|\geqslant\delta|A|\}.
$$
From $\Pr(A,M)\geqslant\epsilon$ we obtain
$|X|/|A|,|Y|/|M|\geqslant(\epsilon-\delta)/(1-\delta)>1/(r+1)$.
The set $X$ is symmetric and invariant under conjugation by $A$, while
$Y$ is symmetric and $A$-invariant. Hence, with
$A_0=\langle X\rangle$ and $L=\langle Y\rangle$, we have
$A_0\unlhd A$, $|A:A_0|\leqslant r$, $L^A=L$ and $|M:L|\leqslant r$.

By Eberhard's lemma,
$$
A_0=X^{3r},\qquad L=Y^{3r},\qquad
|M:C_M(a)|,\ |A:C_A(x)|\leqslant\delta^{-3r}
$$
for every $a\in A_0$ and $x\in L$, respectively. Let
$M_0=\operatorname{core}_M(L)$. Then
$|M:M_0|\leqslant r!$, and $M_0$ is still $A$-invariant; hence it is a
strong left ideal of $B$.

Since both $|a_0^{M_0}| \le |a_0^M|$ and $|m_0^{A_0}| \le |m_0^A|$ are $r$-bounded  for all $m_0 \in M_0$ and $a_0 \in A_0$, we can apply Lemma~\ref{lem:mixed-BFC} to the group $M_0A_0$. Thus the subgroup
$K=[M_0,A_0]$ has $\epsilon$-bounded order and is normal in
$M_0A_0$. Since
$|G:M_0A_0|\leqslant r!r$, the standard consequence of Dietzmann's
lemma shows that its normal closure
$N=\langle K^G\rangle$
has $\epsilon$-bounded order. As $M_0\unlhd G$, we have $N\leqslant
M_0$; moreover $N\unlhd G$, so $N$ is a strong left ideal. Finally,
$A_0*M_0\leqslant N$. Thus $A_0$ lies in the kernel of the induced
action on $M_0/N$, and its image has order at most $|A:A_0|\leqslant r$.
Compare also the normal-closure observation in
\cite[Remark 2.6]{DS22}.
\end{proof}

For $\Pt(B)$ the conclusion can be improved. The reason is that a pair
counted by $\Pt(B)$ satisfies both $a*b=1$ and $b*a=1$. We first isolate
the elementary action-theoretic reduction which will allow us to exploit
this symmetry.

\begin{lem}\label{lem:fixed-point-pruning}
Let a finite group $A$ act on a finite group $G$, and let
$N\leqslant M\unlhd G$ be $A$-invariant normal subgroups. Assume that
the image of $A$ in $\Aut(M/N)$ has order at most $c$. For every $\eta>0$
there is an $A$-invariant normal subgroup $L$ with $N\leqslant L\leqslant
M$ and $|G:L|$ is $(c,\eta,|G:M|)$-bounded such that every non-trivial
element of the image of $A$ on $L/N$ fixes fewer than $\eta|L/N|$
elements.
\end{lem}
\begin{proof}
Let $\rho:A\to\Aut(M/N)$ be the given action, and set $Q=\rho(A)$. Thus $|Q|\leqslant c$. We argue by induction on $|Q|$, assuming that the assertion holds whenever the image of the induced action of $A$ on $M_i/N_i$ has order strictly smaller than $|Q|$ for all possible pairs of $A$-invariant subgroups $N_i \leqslant M_i \normaleq G$.

If $|Q|=1$, we may take $L=M$. More generally, if every non-trivial element $q\in Q$ satisfies $|C_{M/N}(q)|<\eta|M/N|$, then again we may take $L=M$. Hence suppose that there exists $1\neq q\in Q$ such that $|C_{M/N}(q)|\geqslant\eta|M/N|$.

Let $R=\langle q\rangle^Q \; \unlhd \; Q$ be the normal closure of $q$ in $Q$, and  $C=\bigcap_{u\in Q}C_{M/N}(q^u)$. Conjugation by $Q$ permutes the subgroups $C_{M/N}(q^u)$, so $C$ is $Q$-invariant. Moreover, $C=C_{M/N}(R)$, since $R$ is generated by the conjugates of $q$. In particular, $R$ acts trivially on $C$.

For every $u\in Q$, the elements $q$ and $q^u$ have fixed-point subgroups of the same order, and hence $[M/N:C_{M/N}(q^u)]\leqslant\eta^{-1}$. Since $q$ has at most $|Q|\leqslant c$ distinct conjugates, we obtain
$$
[M/N:C]
\leqslant
\prod_{q^u\in q^Q}[M/N:C_{M/N}(q^u)]
\leqslant \eta^{-c}.
$$

Let $\pi:M\to M/N$ be the natural projection, and let $\widetilde C=\pi^{-1}(C)$. Since $\pi$ is $A$-equivariant and $C$ is $Q$-invariant, $\widetilde C$ is $A$-invariant. Moreover, $N\leqslant\widetilde C\leqslant M$ and $[M:\widetilde C]=[M/N:C]\leqslant\eta^{-c}$. Thus $[G:\widetilde C]\leqslant |G:M|\eta^{-c}$.

Set $M_1=\operatorname{core}_G(\widetilde C)=\bigcap_{g\in G}\widetilde C^{\,g}$. Since $N\unlhd G$ and $N\leqslant\widetilde C$, every conjugate of $\widetilde C$ contains $N$, and hence $N\leqslant M_1$. By construction $M_1\unlhd G$. Also $M_1$ is $A$-invariant: indeed, for every $a\in A$, using $\widetilde C^a=\widetilde C$ and the fact that $g\mapsto g^a$ permutes $G$, we have
$$
M_1^a
=
\left(\bigcap_{g\in G}\widetilde C^{\,g}\right)^a
=
\bigcap_{g\in G}(\widetilde C^a)^{\,g^a}
=
\bigcap_{h\in G}\widetilde C^{\,h}
=
M_1.
$$
Furthermore, $|G:M_1|\leqslant |G:\widetilde C|!$, by the usual bound for the core of a subgroup. Consequently, $|G:M_1|$ is $(c,\eta,|G:M|)$-bounded.

It remains to check that the image of $A$ on $M_1/N$ has order strictly smaller than $|Q|$. Since $M_1\leqslant\widetilde C$, we have $M_1/N\leqslant C$, so $R$ acts trivially on $M_1/N$. Let $\rho_1:A\to\Aut(M_1/N)$ be the induced action and let $Q_1=\rho_1(A)$. Since $\ker\rho\leqslant\ker\rho_1$, the map $\rho_1$ factors through $Q$, giving a surjective homomorphism $\overline\rho_1:Q\to Q_1$. Since $R$ acts trivially on $M_1/N$ and $R\neq 1$, we have
$$
R\leqslant\ker\overline\rho_1,
\qquad
|Q_1|\leqslant |Q/R|<|Q|.
$$

We may therefore apply the induction hypothesis to $N\leqslant M_1\unlhd G$. It gives an $A$-invariant normal subgroup $L$ with $N\leqslant L\leqslant M_1\leqslant M$ such that every non-trivial element of the image of $A$ on $L/N$ fixes fewer than $\eta|L/N|$ elements. Moreover, $|G:L|$ is bounded in terms of $c$, $\eta$, and $|G:M_1|$, and hence is $(c,\eta,|G:M|)$-bounded. This completes the induction.
\end{proof}

We can now exploit the full definition of $\Pt(B)$. The resulting
$1/\sqrt{\epsilon}$ bound will be shown to be best possible.

\begin{thm}\label{thm:Pt-optimal}
Let $B$ be a finite skew brace with $\Pt(B)\geqslant\epsilon>0$, and put
$s=\lfloor1/\sqrt{\epsilon}\rfloor$. There exist
$A_0\unlhd(B,\circ)$ and strong left ideals $N\leqslant M_0\leqslant B$
such that
$$
|(B,\circ):A_0|\leqslant s,\quad
|(B,+):M_0|\text{ is $\epsilon$-bounded},\quad
|N|\text{ is $\epsilon$-bounded},\quad
A_0*M_0\leqslant N.
$$
Equivalently, the induced lambda action on $M_0/N$ has image of order
at most $s$.
\end{thm}

\begin{proof}
We apply Theorem~\ref{thm:Plambda-optimal-form} to obtain strong left
ideals $N\leqslant M_0$ of $B$ such that $|N|$ and $|(B,+):M_0|$ are
$\epsilon$-bounded and the image of the action of $(B,\circ)$ on $M_0/N$ has $\epsilon$-bounded order.

Set
$$
\eta=\frac14\left(\epsilon-\frac1{(s+1)^2}\right)>0.
$$
Apply Lemma~\ref{lem:fixed-point-pruning} and replace $M_0$ by the strong left ideal $L$ that it supplies. Thus $|(B,+):L|$ remains $\epsilon$-bounded, and every non-trivial element of the resulting image $Q$ (as in the proof of Lemma \ref{lem:fixed-point-pruning}) fixes fewer than $\eta|L/N|$ elements of $L/N$,  that is, $|\C_{L/N}(a)|<\eta|L/N|$. This implies $|\C_{L}(a)|<\eta|L|$.  Let
$A_0$ be the kernel of this action and write $q=|Q|$.

If $a\notin A_0$, then $|C_{L}(a)|<\eta|L|$. Since $L$ is
$\lambda_a$-invariant, the fixed points of $\lambda_a$ in any additive
coset of $L$, if non-empty, form a coset of $C_{L}(a)$. Hence
$$
|\{x\in B:a*x= 1\}|<\eta|B|
\qquad(a\notin A_0).
$$
There are at most $|B|^2/q^2$ pairs $(a,b)$ with $a,b\in A_0$. If
$a\notin A_0$, a pair counted by $\Pt(B)$ must satisfy $a*b=1$, and if
$a\in A_0$, $b\notin A_0$, it must satisfy $b*a=1$. Therefore
$$
\Pt(B)<\frac1{q^2}+2\eta.
$$
If $q\geqslant s+1$, the right-hand side is smaller than $\epsilon$,
contrary to the hypothesis. Thus $q\leqslant s$. Taking
$M_0=L$ gives the assertion.
\end{proof}

The theorem also yields a large multiplicative subgroup on which all
star-products are contained in the same bounded obstruction.

\begin{cor}
Under the hypotheses of Theorem~\ref{thm:Pt-optimal}, there exist
$D\leqslant(B,\circ)$ and a strong left ideal $N$ such that
$|(B,\circ):D|$ and $|N|$ are $\epsilon$-bounded,
and  $D*D\leqslant N$.
\end{cor}

\begin{proof}
Take $A_0,M_0,N$ as in Theorem~\ref{thm:Pt-optimal} and put
$D=A_0\cap(M_0,\circ)$. Since a left ideal has the same additive and
multiplicative index, $|(B,\circ):D|$ is $\epsilon$-bounded. For
$x,y\in D$, we have $x\in A_0$ and $y\in M_0$, and hence $x*y\in N$.
\end{proof}

For the three classes treated in the preceding section, the conclusion
can be strengthened from strong left ideals to ideals.

\begin{cor} 
Let $B$ be a finite skew brace with $\Pt(B)\geqslant\epsilon>0$ and suppose that $B$ lies in the class $\mathcal{T}$. Then there exist ideals $B_2\leqslant B_1\trianglelefteq B$ such that $B_1/B_2$ is a trivial skew brace, and both $|B_2|$ and $|B:B_1|$ are $\epsilon$-bounded.
\end{cor}

\begin{proof}
Let $N\leqslant M_0$ be the strong left ideals supplied by Theorem~\ref{thm:Pt-optimal}, and let
$K$ be the kernel of the induced action of $(B,\circ)$ on $M_0/N$. Then
$|(B,\circ):K|\leqslant\lfloor1/\sqrt{\epsilon}\rfloor$, $|(B,+):M_0|$ and $|N|$ are $\epsilon$-bounded, and $K*M_0\leqslant N$.

We first observe that $K$ is a sub-skew brace of $B$. This is immediate in the $\lambda$-homomorphic case, since $\lambda_{a+b}=\lambda_a\lambda_b$. In the symmetric case we have $\lambda_{a+b}=\lambda_b\lambda_a$, so the same conclusion holds. Finally, suppose that $B$ is two-sided. Skew right distributivity gives
$$
(a+b)*x=-b+(a*x)+b+(b*x).
$$
Since $N\unlhd(B,+)$, this shows that $a,b\in K$ implies $a+b\in K$. Applying the same identity to $a+(-a)=1$ shows that $-a\in K$ whenever $a\in K$. Thus $K$ is also a subgroup of $(B,+)$. Since it is already a subgroup of $(B,\circ)$, being the kernel of the induced lambda action, $K$ is a sub-skew brace in all three cases.

Put $D=K\cap M_0$. Then $D$ is a sub-skew brace and
$$
|B:D|
\leqslant |(B,\circ):K|\,|(B,\circ):M_0|
=
|(B,\circ):K|\,|(B,+):M_0|,
$$
so $|B:D|$ is $\epsilon$-bounded. Moreover, since $D\leqslant K\cap M_0$, we have $D*D\leqslant N$.
Apply the corresponding finite-index ideal-core lemma, namely Lemma~\ref{lem:finite-index-ideal-core-two-sided}, Lemma~\ref{lem:finite-index-ideal-core-symmetric}, or Lemma~\ref{lem:finite-index-ideal-core-lambda-homomorphic}, we obtain an ideal $B_1$ of $B$ such that $B_1\leqslant D$ and $|B:B_1|$ is $\epsilon$-bounded.

Set $N_1=N\cap B_1$. Since $B_1\leqslant D$, we have $B_1*B_1\leqslant N_1$. Moreover, $N_1$ is an ideal of $B_1$. Indeed, it is normal in $(B_1,+)$ and invariant under the lambda action of $B_1$. Also, since $B_1*B_1\leqslant N_1$, the map
$x\mapsto x+N_1$ from $(B_1,\circ)$ to $(B_1/N_1,+)$ is a group homomorphism with kernel $N_1$. Hence $N_1\unlhd(B_1,\circ)$. Notice also that $|N_1|\leqslant|N|$, and therefore $|N_1|$ is $\epsilon$-bounded.

Let $B_2$ be the ideal of $B$ generated by $N_1$. In the symmetric and $\lambda$-homomorphic cases, Lemma~\ref{lem:bounded-ideal-closure-symmetric} and Lemma~\ref{lem:bounded-ideal-closure-lambda-homomorphic}, respectively, show that $|B_2|$ is $\epsilon$-bounded. In the two-sided case, the proof of Lemma~\ref{lem:bounded-ideal-closure-two-sided} applies verbatim with an arbitrary ideal of $B_1$ in place of $\Gamma_2(B_1)$: the only property used there is that the subgroup in question is an ideal of a sub-skew brace of bounded index. Hence $|B_2|$ is $\epsilon$-bounded also in this case.

Since $B_1$ is an ideal of $B$ containing $N_1$, the ideal $B_2$ generated by $N_1$ satisfies $B_2\leqslant B_1$. Finally,
$$
B_1*B_1\leqslant N_1\leqslant B_2,
$$
and therefore the induced lambda action on $B_1/B_2$ is trivial. Thus the two operations on $B_1/B_2$ coincide, so $B_1/B_2$ is a trivial skew brace.
\end{proof}

\begin{rem}
The conclusion cannot in general be strengthened by requiring $B_1/B_2$ to be abelian, or equivalently to be a trivial brace in the terminology used here. Indeed, let $G_n=A_n$, for $n\geqslant5$, and regard $G_n$ as the trivial skew brace $B_n=(G_n,+,\circ)$ with $+=\circ$. Then $\Pt(B_n)=1$, and $B_n$ is two-sided, symmetric and $\lambda$-homomorphic. Since $A_n$ is non-abelian simple, the only ideals of $B_n$ are $\{1\}$ and $B_n$. Thus, if $B_2\leqslant B_1$ are ideals and $B_1/B_2$ is abelian, either $B_1=\{1\}$, in which case $|B_n:B_1|=|A_n|$, or $B_2=B_n$, in which case $|B_2|=|A_n|$. Since $|A_n|\to\infty$, neither quantity can be bounded in terms of $\epsilon$ alone. Hence ``trivial skew brace'' cannot be replaced by ``trivial brace'' in the corollary.
\end{rem}

We now show that the $1/\sqrt{\epsilon}$ term in
Theorem~\ref{thm:Pt-optimal} is optimal. The same family also shows
that the lower bound $|S(B)|\geqslant\sqrt{\epsilon}|B|$ is
asymptotically best possible.

\begin{pro} 
For every prime $q$ there exists a family $(B_n)$ with
$\Pt(B_n)>1/q^2$ such that, for every fixed $C,D$, if
$N\leqslant M$ are strong left ideals with $|B_n:M|\leqslant C$ and
$|N|\leqslant D$, then for all sufficiently large $n$ the induced
lambda action on $M/N$ has image of order at least $q$. Moreover,
$$
\lim_{n\to\infty}
\frac{|\{x\in B_n:x*x= 1\}|}{|B_n|}=\frac1q.
$$
\end{pro}

\begin{proof}
Choose a finite field $F$ containing an element $\alpha$ of order $q$,
put $V_n=F^n$, and let $B_n=V_n\times C_q$ with
$$
(v,i)+(w,j)=(v+w,i+j),\qquad
(v,i)\circ(w,j)=(v+\alpha^iw,i+j).
$$
Then $(v,i)*(w,j)=((\alpha^i-1)w,0)$. Since $q$ is prime,
$\alpha^i-1$ is invertible for $i\neq0$. By
Lemma~\ref{lem:two-generated-trivial}, a pair
$(v,i),(w,j)$ generates a trivial sub-skew brace precisely when
$i=j=0$ or $v=w=0$. Hence
$$
\Pt(B_n)=\frac1{q^2}+
\left(1-\frac1{q^2}\right)|V_n|^{-2}>\frac1{q^2}.
$$

Let $N\leqslant M$ satisfy the given hypotheses, and set
$W :=M\cap(V_n\times\{0\})$. Then $|V_n:W|\leqslant C$. If some
$0\neq i\in C_q$ acts trivially on $M/N$, then
$(\alpha^i-1)W\leqslant N$. Since $\alpha^i-1$ is invertible,
$|N|\geqslant|W|\geqslant|V_n|/C$, which is impossible for large $n$.
Thus the copy of $C_q$ acts faithfully on $M/N$.

Finally,
$$
\frac{|\{x:x*x= 1\}|}{|B_n|}
=\frac{|V_n|+q-1}{q|V_n|}\longrightarrow\frac1q.
$$
Taking $\epsilon=1/q^2$ proves sharpness of both bounds.
\end{proof}

The passage to the large strong left ideal $M_0$ is also necessary.
Even for a fixed positive value of $\Pt(B)$, keeping the whole additive
group may leave an arbitrarily large lambda image.

\begin{pro} 
For every prime $p$ there is a family $(B_n)$ with
$\Pt(B_n)>1/p^2$ such that, for every fixed $D$ and every strong left
ideal $N$ with $|N|\leqslant D$, the image of
$(B_n,\circ)$ on $(B_n,+)/N$ tends to infinity with $n$. On the other
hand, each $B_n$ has a strong left ideal of index $p$ on which the
lambda action is trivial.
\end{pro}

\begin{proof}
Put $Z_n=C_n=\mathbb F_p^n$, $U_n=\mathbb F_p\times Z_n$, and let
$c\in C_n$ act on $U_n$ by
$\varphi_c(t,z)=(t,z+tc)$. On $B_n=U_n\times C_n$ take the associated
semidirect-product skew brace. Then
$$
(t,z,c)*(t',z',d)=(0,t'c,0).
$$
The four conditions of Lemma~\ref{lem:two-generated-trivial} reduce to
$tc=t'c=td=t'd=0$. Hence either $t=t'=0$ or $c=d=0$, and
$$
\Pt(B_n)=\frac1{p^2}+
\left(1-\frac1{p^2}\right)p^{-2n}>\frac1{p^2}.
$$

If $c$ belongs to the kernel of the action on $B_n/N$, applying it to
$(1,0,0)$ gives $(0,c,0)\in N$. Thus the kernel inside $C_n$ has order
at most $|N|$, and the image has order at least $p^n/|N|$.

In contrast,
$M_n=\{(0,z,c):z\in Z_n,\ c\in C_n\}$ is a strong left ideal with
$|B_n:M_n|=p$, and every star-product with second entry in $M_n$ is
zero.
\end{proof}

The bounded strong left ideal $N$ cannot be omitted either. Passing to
a large strong left ideal does not in general make the lambda action
itself bounded.

\begin{pro} 
For every prime $p$ there is a family $(B_n)$ with
$\Pt(B_n)>1/p^4$ such that, for every fixed $C$, the lambda image on
every strong left ideal $M$ with $|B_n:M|\leqslant C$ tends to infinity
with $n$. Nevertheless, all star-products of $B_n$ lie in a strong
left ideal of order $p$.
\end{pro}
\begin{proof}
Let $W_n=C_n=\mathbb F_p^n$ and let $Z=\mathbb F_p$. Fix a non-degenerate bilinear form
$\langle-,-\rangle:C_n\times W_n\to\mathbb F_p$. Put $U_n=W_n\times Z$ and, for $c\in C_n$, define
$\varphi_c\in\Aut(U_n,+)$ by
$\varphi_c(w,z)=(w,z+\langle c,w\rangle)$. Since the form is bilinear, we have
$\varphi_{c+d}=\varphi_c\varphi_d$, so $C_n$ acts on $U_n$ by automorphisms.

Let $B_n=U_n\times C_n$ be the associated semidirect-product skew brace. Explicitly, its two operations are
$$
(w,z,c)+(w',z',d)=(w+w',z+z',c+d),\;
(w,z,c)\circ(w',z',d)
=(w+w',z+z'+\langle c,w'\rangle,c+d).
$$
Thus the additive group of $B_n$ is the elementary abelian group
$W_n\times Z\times C_n$, while the multiplicative group is the semidirect product
$U_n\rtimes C_n$. In particular,
$|B_n|=p^{2n+1}$.

For $x=(w,z,c)$ and $y=(w',z',d)$, the definition of the lambda map gives
$$
\lambda_x(y)
=-x+(x\circ y)
=(w',z'+\langle c,w'\rangle,d),
\qquad
x*y=\lambda_x(y)-y=(0,\langle c,w'\rangle,0).
$$
Consequently,
$x*x=(0,\langle c,w\rangle,0)$,
$x*y=(0,\langle c,w'\rangle,0)$,
$y*x=(0,\langle d,w\rangle,0)$, and
$y*y=(0,\langle d,w'\rangle,0)$.
By Lemma~\ref{lem:two-generated-trivial}, the pair $x,y$ generates a trivial sub-skew brace if and only if all four of these star-products vanish. Equivalently,
$$
\langle c,w\rangle=\langle c,w'\rangle
=\langle d,w\rangle=\langle d,w'\rangle=0.
$$
Thus, once $c,d\in C_n$ are fixed, both $w$ and $w'$ must belong to the common orthogonal space
$\operatorname{span}\{c,d\}^{\perp}\leqslant W_n$. The coordinates $z,z'\in Z$ are completely unrestricted. We now count the pairs according to
$r=\dim\operatorname{span}\{c,d\}$.

Suppose first that $r=0$. Then $c=d=0$. There are $p^n$ choices for each of $w$ and $w'$, and $p$ choices for each of $z$ and $z'$. Hence the number of pairs in this case is
$N_0=p^{2n+2}$.

Next suppose that $r=1$. The number of one-dimensional subspaces of $C_n$ is
$(p^n-1)/(p-1)$. For each such subspace $L$, the ordered pairs
$(c,d)\in L^2$ which span $L$ are precisely all pairs except $(0,0)$, so there are
$p^2-1$ of them. Therefore the number of ordered pairs $(c,d)$ with
$\dim\operatorname{span}\{c,d\}=1$ is
$(p^n-1)(p+1)$. By non-degeneracy of the bilinear form,
$\operatorname{span}\{c,d\}^{\perp}$ has dimension $n-1$, and hence there are
$p^{n-1}$ choices for each of $w$ and $w'$. Again $z$ and $z'$ are arbitrary. Thus
$N_1=(p^n-1)(p+1)p^{2n}$.

Finally, suppose that $r=2$. There are $p^n-1$ choices for $c\neq0$, and, once $c$ is fixed, there are $p^n-p$ choices for $d$ outside the one-dimensional subspace $\langle c\rangle$. Hence there are
$(p^n-1)(p^n-p)$ ordered pairs $(c,d)$ spanning a two-dimensional subspace. In this case
$\operatorname{span}\{c,d\}^{\perp}$ has dimension $n-2$, so there are
$p^{n-2}$ choices for each of $w$ and $w'$, and $p^2$ choices for $(z,z')$. Therefore
$N_2=(p^n-1)(p^n-p)p^{2n-2}$.
Notice that when $n=1$ this number is zero, as it should be.

It follows that
$$
\begin{aligned}
\Pt(B_n)
&=\frac{N_0+N_1+N_2}{|B_n|^2}\\
&=\frac{p^{2n+2}+(p^n-1)(p+1)p^{2n}
 +(p^n-1)(p^n-p)p^{2n-2}}{p^{4n+2}}\\
&= \frac{1}{p^{2n}}
 +\frac{(p+1)(p^n-1)}{p^{2n+2}}
 +\frac{(p^n-1)(p^n-p)}{p^{2n+4}}\\
&=\frac1{p^4}
 +\frac{(p-1)(p+1)^2}{p^{n+4}}
 +\frac{p(p-1)^2(p+1)}{p^{2n+4}}
>\frac1{p^4}.
\end{aligned}
$$

We next prove the assertion concerning the lambda image. Let $M_0$ be a strong left ideal of $B_n$ such that $|B_n:M_0|\leqslant C$, and set
$\widetilde W_n=W_n\times\{0\}\times\{0\}$ and
$W_0=M_0 \cap \widetilde W_n$. Since $M_0$ is an additive subgroup of $B_n$,
$$
|\widetilde W_n:W_0|
=|\widetilde W_n:\widetilde W_n\cap M_0|
=|\widetilde W_n+M_0:M_0|
\leqslant |B_n:M_0|
\leqslant C.
$$
We identify $W_0$ with the corresponding subspace of $W_n$.

Consider the action on $M_0$ of the subgroup
$\{(0,0,c):c\in C_n\}$ of $(B_n,\circ)$. Since $M_0$ is a strong left ideal, it is invariant under every lambda map, and hence we obtain a homomorphism
$\psi_{M_0}:C_n\to\Aut(M_0)$ defined by
$\psi_{M_0}(c)=\lambda_{(0,0,c)}|_{M_0}$.
If $c\in\ker\psi_{M_0}$, then for every $w\in W_0$ we have
$$
(w,0,0)
=\lambda_{(0,0,c)}(w,0,0)
=(w,\langle c,w\rangle,0),
$$
and therefore $\langle c,w\rangle=0$. Hence
$\ker\psi_{M_0}\leqslant W_0^\perp$.

Since the pairing $C_n\times W_n\to\mathbb F_p$ is non-degenerate,
$\dim W_0^\perp=n-\dim W_0$, and consequently
$$
|\ker\psi_{M_0}|
\leqslant |W_0^\perp|
=p^{n-\dim W_0}
=|W_n:W_0|
\leqslant C.
$$
It follows that
$|\operatorname{Im}\psi_{M_0}|
=|C_n:\ker\psi_{M_0}|
\geqslant p^n/C$.
The image of $\psi_{M_0}$ is contained in the full lambda image of $B_n$ on $M_0$, so the latter also has order at least $p^n/C$. For fixed $C$, this tends to infinity with $n$, uniformly over all strong left ideals $M_0$ satisfying $|B_n:M_0|\leqslant C$.

Finally, put
$N_n=\{(0,z,0):z\in\mathbb F_p\}$. Clearly $|N_n|=p$. Since $(B_n,+)$ is abelian, $N_n$ is normal in $(B_n,+)$, and for every $(w,z,c)\in B_n$ and $t\in\mathbb F_p$ we have
$\lambda_{(w,z,c)}(0,t,0)=(0,t,0)$. Thus $N_n$ is a strong left ideal of $B_n$. Moreover, from the formula for the star-product,
$(w,z,c)*(w',z',d)=(0,\langle c,w'\rangle,0)\in N_n$
for all elements of $B_n$. Hence every star-product of $B_n$ is contained in the strong left ideal $N_n$ of order $p$. This completes the proof.
\end{proof}

Finally, even after allowing both a large multiplicative subgroup and a
large additive subgroup, the bounded obstruction cannot in general be
chosen to be an ideal. Nor can the two large pieces be replaced by one
large sub-skew brace which becomes trivial after factoring out a bounded
sub-skew brace.

\begin{pro} 
There exists a family $(B_n)$ with $\Pt(B_n)>1/27$ for which both of the
following fail uniformly in $n$:
\begin{enumerate}
\item there are $B_2\leqslant B_1\leqslant B_n$, with $B_1/B_2$
trivial, such that $|B_n:B_1|$ and $|B_2|$ are bounded;
\item there are $A_0\leqslant(B_n,\circ)$, $M_0\leqslant(B_n,+)$ and
an ideal $I$ such that $|B_n:A_0|$, $|B_n:M_0|$ and $|I|$ are bounded
and $A_0*M_0\leqslant I$.
\end{enumerate}
\end{pro}

\begin{proof}
Recall the family
$B_{n,m}=U_m\times V_n\times\mathbb F_3$ constructed in the preceding
section, where $U_m=\mathbb F_7^m$, $V_n=\mathbb F_3^{2n}$ carries a
non-degenerate alternating form $\beta$, and
$$
(u,v,z)*(u',w,t)=((2^z-1)u',0,-\beta(v,w)).
$$
Lemma~\ref{lem:two-generated-trivial} shows that a pair generates a
trivial sub-skew brace precisely when $\beta(v,w)=0$ and
$(2^z-1)u=(2^z-1)u'=(2^t-1)u=(2^t-1)u'=0$. Hence
$$
\Pt(B_{n,m})=
\left(\frac1{27}+\frac2{3^{2n+3}}\right)
\left(1+\frac8{7^{2m}}\right)>\frac1{27}.
$$

Set $m=n$ and write $B_n=B_{n,n}$. The argument in the preceding
section shows that whenever $B_2\leqslant B_1\leqslant B_n$ and
$B_1/B_2$ is trivial,
$$
|B_n:B_1|\geqslant3^{n+1}
\quad\text{or}\quad
|B_2|\,|B_n:B_1|\geqslant7^n.
$$
This proves (1).

For (2), suppose $|B_n:A_0|,|B_n:M_0|\leqslant C$ and
$A_0*M_0\leqslant I$. Let $L_A,L_M$ be the images of $A_0,M_0$ in
$V_n$. Then $|V_n:L_A|,|V_n:L_M|\leqslant C$. For sufficiently large
$n$ they cannot be orthogonal, since otherwise
$\dim L_A+\dim L_M\leqslant2n$, whereas both have codimension at most
$\log_3C$. Thus there are
$a=(u,v,z)\in A_0$ and $x=(u',w,t)\in M_0$ with $\beta(v,w)\neq0$.
Consequently $a*x=(r,0,c)\in I$ for some $c\neq0$.

For every $s\in U_n$, normality of $I$ in $(B_n,\circ)$ gives
$(r+(1-2^c)s,0,c)\in I$. Since also $(r,0,c)\in I$ and $I$ is an
additive subgroup, $((1-2^c)s,0,0)\in I$. As $2$ has order $3$ in
$\mathbb F_7^\times$, $1-2^c\neq0$; hence $U_n\leqslant I$ and
$|I|\geqslant7^n$. This proves (2), and the proof is complete.
\end{proof}

These examples identify the role of each term in
Theorem~\ref{thm:Pt-optimal}. The order $1/\sqrt{\epsilon}$ of the
residual lambda image and the density $\sqrt{\epsilon}$ of
self-trivial elements are optimal. In general one must pass to a
strong left ideal of bounded index and one must also factor out a
strong left ideal of bounded order. The latter cannot be replaced by
a bounded ideal, and the two large pieces cannot be replaced by a
single large sub-skew brace which is trivial modulo a bounded
sub-skew brace.

Thus the structural information carried by $\Pt(B)$ concerns precisely
the interaction between the two operations. No corresponding restriction
on the internal structure of $(B,+)$ or $(B,\circ)$ can hold without
additional hypotheses.

We conclude with some remarks on the bounds on $P_{\lambda}(B)$. Define
$$
\Fix^r_B(x):=\{b\in B\mid \lambda_x(b)=b\}.
$$
Then $\Fix^r_B(x)$ is a subgroup of $(B,+)$. Moreover,
$$
P_{\lambda}(B)
=
\frac{1}{|B|}|\Ker(\lambda)|
+
\frac{1}{|B|^2}
\sum_{x\in B-\Ker(\lambda)}|\Fix^r_B(x)|.
$$

Put
$n:=|B:\Ker(\lambda)|=|\lambda(B)|$. Assume that $|B|>1$, and let $p$ be the smallest prime divisor of $|B|$.
If $x\notin\Ker(\lambda)$, then $\Fix^r_B(x)$ is a proper subgroup of
$(B,+)$. Hence
$$
|\Fix^r_B(x)|\leqslant\frac{|B|}{p}.
$$
Using the preceding expression for $P_{\lambda}(B)$, we obtain the
following observations:
\begin{enumerate}
    \item For the given $n$, we get
    $$
    \frac{1}{n}
    \leqslant
    P_{\lambda}(B)
    \leqslant
    \frac{n+p-1}{np}
    \leqslant
    \frac{n+1}{2n}.
    $$

    \item
    If $n=p$, then
    $$
    \frac{1}{p}
    \leqslant
    P_{\lambda}(B)
    \leqslant
    \frac{2p-1}{p^2}.
    $$
    Moreover, equality in the upper bound holds if and only if
    $$
    |\Fix^r_B(x)|=\frac{|B|}{p}
    $$
    for every $x\in B-\Ker(\lambda)$.

\item
    We have
    $$
    P_{\lambda}(B)=1,\qquad
    P_{\lambda}(B)=\frac{3}{4},\qquad
    P_{\lambda}(B)=\frac{2}{3},
    \qquad\text{or}\qquad
    P_{\lambda}(B)\leqslant\frac{5}{8}.
    $$
    Moreover, $P_{\lambda}(B)=1$ if and only if $B$ is a trivial skew brace,
    while
    $$
    P_{\lambda}(B)=\frac{3}{4}
    $$
    if and only if $n=2$ and
    $$
    |\Fix^r_B(x)|=\frac{|B|}{2}
    $$
    for every $x\in B-\Ker(\lambda)$. Indeed, if $n\geqslant4$, then (1) gives
    $P_{\lambda}(B)\leqslant5/8$. If $n=2$, then all elements of
    $B-\Ker(\lambda)$ induce the same non-trivial automorphism; hence, writing
    $d=|B:\Fix^r_B(x)|$, we have
    $$
    P_{\lambda}(B)=\frac12+\frac{1}{2d}.
    $$
    Thus $d=2$ gives $P_{\lambda}(B) = 3/4$, $d=3$ gives $P_{\lambda}(B) = 2/3$, and $d\geqslant4$ gives
    $P_{\lambda}(B) \leqslant 5/8$. If $n=3$, then the two non-trivial elements of $\lambda(B)$ are inverses of each other
    and therefore have the same fixed-point subgroup. Hence, for some integer
    $d \geqslant 2$,
    $$
    P_{\lambda}(B)=\frac13+\frac{2}{3d},
    $$
    which is $2/3$ when $d=2$ and is at most $5/9<5/8$ when $d\geqslant3$.

    \item
    If $n=p^k$ for some positive integer $k$, then
    $$
    P_{\lambda}(B)
    \leqslant
    \frac{p^k+p-1}{p^{k+1}}.
    $$
    Equality holds if and only if
    $$
    |\Fix^r_B(x)|=\frac{|B|}{p}
    $$
    for every $x\in B-\Ker(\lambda)$.

    \item Conversely,    if
    $$
    P_{\lambda}(B)=\frac{2p-1}{p^2},
    $$
    then $n=p$. In this case
    $$
    |\Fix^r_B(x)|=\frac{|B|}{p}
    $$
    for every $x\in B-\Ker(\lambda)$.
\end{enumerate}

The following easy example shows that $3/4$ is attained by
$P_{\lambda}(B)$. Let $B=\mathbb{Z}_4$ with the usual addition and define
`$\circ$' by
$$
x\circ y=x+y+2xy
\qquad\mbox{for all }x,y\in B,
$$
where $xy$ denotes the usual multiplication. Then $(B,+,\circ)$ is a
skew brace with
$P_{\lambda}(B)=3/4$.


\section{Sylow Local Commuting Probability}

We record some consequences of the recent Sylow theory for skew braces.
Truman (\cite{Tru26}) proved an unconditional first Sylow theorem for finite skew braces. In the two-sided case, however, one has a stronger statement:
 the Sylow $p$-sub-skew braces of a finite
two-sided skew brace coincide with the Sylow $p$-subgroups of the
multiplicative group (\cite[Theorem 14]{CDMFT26}). More generally, in the soluble two-sided case, Hall sub-skew braces coincide with Hall subgroups of the multiplicative group.

We first recall the concept of Sylow local commuting probability for groups defined in \cite{DLMS25}. Let $B$ be a finite skew brace. By the first Sylow theorem for finite skew braces proved in \cite{Tru26}, for every prime $p\in\pi(B)$ there exists a Sylow $p$-sub-skew brace of $B$. For distinct primes
$p,q\in\pi(B)$ and Sylow sub-skew braces $P\in\Syl_p(B)$,
$Q\in\Syl_q(B)$, we write $\Pb(P,Q)$ for the probability that a random
pair $(x,y)\in P\times Q$ brace-commutes. Define
\[
\Pb_{\Syl}(B)=
\min_{\substack{p,q\in\pi(B)\\p\neq q}}
\max\{\Pb(P,Q):P\in\Syl_p(B),\,Q\in\Syl_q(B)\}.
\]
If $|\pi(B)|\leqslant 1$, we set $\Pb_{\Syl}(B)=1$. Since a Sylow $p$-sub-skew brace has order equal to the full $p$-part of $|B|$, it is simultaneously a Sylow $p$-subgroup of both $(B,+)$ and $(B,\circ)$. In the two-sided case, moreover, the Sylow sub-skew braces are precisely the multiplicative Sylow subgroups. We begin with the following easy but useful observation.

\begin{lem}\label{lem:sylow-brace-prob-to-groups}
Let $B$ be a finite skew brace, and let $H,K\leqslant B$ be sub-skew braces.
Then
$\Pb(H,K)\leqslant \Pr((H,+),(K,+))$ and
$\Pb(H,K)\leqslant \Pr((H,\circ),(K,\circ))$.
Consequently, if $\Pb_{\Syl}(B)\geqslant\epsilon$, then
both underlying groups $(B,+)$ and $(B,\circ)$ satisfy the Sylow-local
commuting probability hypothesis with constant $\epsilon$.
\end{lem}

\begin{proof}
If $x\in H$ and $y\in K$ brace-commute, then $[x,y]^+=1$ and
$[x,y]^\circ=1$. Hence every brace-commuting pair is both an additively
commuting pair and a multiplicatively commuting pair. By \cite{Tru26}, Sylow $p$-sub-skew braces exist for every prime $p \mid |B|$. Since such a sub-skew brace has order equal to the full $p$-part of $|B|$, it is a Sylow $p$-subgroup of both underlying groups. Therefore the same Sylow sub-skew braces occurring in the definition of $\Pb_{\Syl}(B)$ give Sylow subgroups in both $(B,+)$ and $(B,\circ)$, and the conclusion follows.
\end{proof}

For a finite group $G$, let $F(G)$ denote the Fitting subgroup and let
$F_2(G)$ denote the second term of the upper Fitting series. For a finite
two-sided skew brace $B$, define $F^+_0(B)= 1$ and, inductively,
$F^+_{i+1}(B)/F^+_i(B)=F((B/F^+_i(B),+))$.

\begin{lem}\label{lem:additive-fitting-ideals-sylow}
Let $B$ be a finite two-sided skew brace. Then $F^+_i(B)$ is an ideal of
$B$ for every $i\geqslant 0$. Moreover, if $J\leqslant I$ are ideals of $B$ and
$(I/J,+)$ is nilpotent, then $I/J$ is a direct product of its Sylow ideals,
and distinct Sylow ideals brace-centralize each other.
\end{lem}

\begin{proof}
The subgroups $F^+_i(B)$ are characteristic in the additive group at each
step. By \cite[Lemma 4.1]{Nas19}, characteristic subgroups of the additive group of a two-sided skew brace are ideals. This proves
the first assertion.

Now assume that $(I/J,+)$ is nilpotent. It is enough to replace $B$ by
$I/J$. Let $B_p$ be the additive Sylow $p$-subgroup of $(B,+)$. Since
$(B,+)$ is nilpotent, $B_p$ is characteristic in $(B,+)$, hence an ideal of
$B$. It is therefore a Sylow $p$-sub-skew brace. By the two-sided Sylow
theorem of \cite[Theorem 14]{CDMFT26}, the Sylow sub-skew braces are
precisely the multiplicative Sylow subgroups, so these ideals are also the
multiplicative Sylow subgroups. Thus, $B$ is the product of the ideals
$B_p$.

If $p\neq q$, $x\in B_p$ and $y\in B_q$, then $[x,y]^+$ lies in both
$B_p$ and $B_q$, hence is trivial. Also $x*y$ lies in $B_q$, because $B_q$
is a left ideal, and it lies in $B_p$, because the image of $x$ in
$B/B_p$ is trivial. Thus $x*y= 1$. Similarly $y*x= 1$. Hence distinct Sylow
ideals brace-centralize each other.
\end{proof}

\begin{thm}
\label{thm:sylow-fitting-two-sided-strong}
Let $B$ be a finite two-sided skew brace. If $\Pb_{\Syl}(B)\geqslant\epsilon>0$,
then there are canonical ideals $F\leqslant I$ of $B$ such that $[B:I]$ is
$\epsilon$-bounded, and both $F$ and $I/F$ are direct products of their
Sylow ideals. More precisely, one may take
$F=F^+_1(B)=F((B,+))$ and $I=F^+_2(B)=F_2((B,+))$.

Moreover, the multiplicative group also has bounded Fitting height up to
bounded index: $F_2((B,\circ))$ has $\epsilon$-bounded index in
$(B,\circ)$.
\end{thm}

\begin{proof}
By Lemma~\ref{lem:sylow-brace-prob-to-groups}, the additive group $(B,+)$
satisfies the Sylow-local commuting probability hypothesis with constant
$\epsilon$. Hence, by one of the main theorems of \cite{DLMS25}, the subgroup $F_2((B,+))$ has
$\epsilon$-bounded index in $(B,+)$. Therefore $I=F^+_2(B)$ has
$\epsilon$-bounded index in $B$. By Lemma~\ref{lem:additive-fitting-ideals-sylow},
$F=F^+_1(B)$ and $I=F^+_2(B)$ are ideals.

By construction, $(F,+)$ is nilpotent and $(I/F,+)$ is nilpotent. Applying
Lemma~\ref{lem:additive-fitting-ideals-sylow} to $F$ and to $I/F$, we get
that both are direct products of their Sylow ideals, with distinct Sylow
ideals brace-centralising each other.

The same argument, now applied to the multiplicative group, gives that
$F_2((B,\circ))$ has $\epsilon$-bounded index in $(B,\circ)$, because in
the two-sided case the Sylow sub-skew braces are exactly the multiplicative
Sylow subgroups.
\end{proof}

We also record the Hall version, which is stronger when the two underlying
groups are soluble. In the two-sided soluble case, Hall sub-skew braces
coincide with multiplicative Hall subgroups by \cite[Theorem 14]{CDMFT26}.

\begin{thm} 
Let $B$ be a finite two-sided skew brace such that both $(B,+)$ and
$(B,\circ)$ are soluble. Suppose that, for every prime $p\in\pi(B)$, there
exist a Sylow $p$-sub-skew brace $P$ and a Hall $p'$-sub-skew brace $H$
such that $\Pb(P,H)\geqslant\epsilon>0$. Then $F^+_1(B)=F((B,+))$ is an ideal
of $\epsilon$-bounded index in $B$, and it is a direct product of its
Sylow ideals. Moreover, $F((B,\circ))$ has $\epsilon$-bounded index in
$(B,\circ)$.
\end{thm}

\begin{proof}
By Lemma~\ref{lem:sylow-brace-prob-to-groups}, for every prime $p$ there
are a Sylow $p$-subgroup and a Hall $p'$-subgroup of the additive group
whose ordinary group commuting probability is at least $\epsilon$. Since
$(B,+)$ is soluble, the Hall version of the theorem of Detomi, Lucchini,
Morigi and Shumyatsky \cite{DLMS25} implies that $F((B,+))$ has
$\epsilon$-bounded index. By Lemma~\ref{lem:additive-fitting-ideals-sylow},
$F^+_1(B)=F((B,+))$ is an ideal of $B$, and since its additive group is
nilpotent, it is a direct product of its Sylow ideals.

Applying the same argument to the multiplicative group, and using that in
the two-sided case Hall sub-skew braces are multiplicative Hall subgroups,
we obtain that $F((B,\circ))$ has $\epsilon$-bounded index in
$(B,\circ)$.
\end{proof}

\begin{rem}
The statements above are deliberately formulated with ideals on the
additive Fitting side. The reason is that characteristic subgroups of
$(B,+)$ are ideals in a two-sided skew brace. The multiplicative conclusions
are kept as group-theoretic shadows because a characteristic subgroup of
$(B,\circ)$ need not be visibly an ideal without an additional invariance
argument. Thus Theorem~\ref{thm:sylow-fitting-two-sided-strong} gives a
canonical ideal of bounded index, while still retaining a simultaneous
multiplicative Fitting bound.
\end{rem}

The following result may be known in the literature, but we include a proof.

\begin{thm}\label{p^n is solvable}
    Let $B$ be a finite skew brace of order $p^n$ for some prime $p$. Then $B$ is a solvable skew brace.
\end{thm}
\begin{proof}
    We prove this by induction on $n\in \mathbb{N}$. For $n=1,2$, we know that $B$ is nilpotent; hence solvable. Assume that every skew brace  of prime power order having order at most $p^{n-1}$ is solvable.  Let $B$ be a skew brace of order $p^n$. Then $B^2=\langle b\ast c \mid b,c\in B\rangle^+$ is an ideal of $B$.
    If $B^2=\{1\}$, then $B$ is a trivial skew brace whose common underlying group is a finite $p$-group, and hence $B$ is solvable. Assume that $B^2 \neq \{1\}$. By \cite[Proposition 4.4.]{CSV19}, $B$ is left nilpotent, which implies that $B\neq B^2$. Hence,  $B^2$ is a nontrivial proper ideal of $B$. Thus the orders of the skew braces $B/B^2$ and $B^2$ are $\leqslant p^{n-1}$. By the induction hypothesis, both $B/B^2$ and $B^2$ are solvable, which proves that $B$ is  solvable. The induction argument now completes the proof.
\end{proof}

\begin{pro}\label{pn, uniqe ideal}
     Let $(B,+,\circ)$ be a finite skew brace of order $pn$ for some prime $p$ and $n\in \mathbb{N}$ such that $p>n$. Then $B$ has a unique ideal of order $p$. 
\end{pro}
\begin{proof}
    By \cite{Tru26}, $B$ admits  a Sylow sub-skew brace $H$ of order $p$. By Sylow theorem for groups, $H$ is a normal subgroup of both groups $(B,+)$ and $(B,\circ)$. So $H$ is a unique subgroup of $(B,+)$ having order $p$, which shows that $\lambda_b(H)=H$ for all $b\in B$. Hence $B$ has a unique ideal $H$ of order $p$. 
\end{proof}

As a direct consequence, we get
\begin{cor}\label{pq solvable}
   A skew brace $(B, +, \circ)$ is solvable if (i) $|B|= pq$, where $p, q$ are primes; (ii) both groups $(B,+)$ and $(B,\circ)$ are nilpotent. 
\end{cor}
\begin{proof}
    If $p=q$, then by Theorem \ref{p^n is solvable}, $B$ is a solvable skew brace. If $p>q$, then by Proposition \ref{pn, uniqe ideal}, $B$ has a unique ideal $H$ of order $p$. Then $H$ and $B/H$ are both skew braces of prime order, and hence are solvable, proving that $B$ is solvable.  Now assume that both groups $(B,+)$ and $(B,\circ)$ are nilpotent. Then, by \cite[Corollary 4.3.]{CSV19}, $B$ is a direct product of its Sylow ideals, say, $I_1\times I_2\times \cdots \times I_k $, where $|I_j| = p_j^{n_j}$ for distinct primes  $p_j$. By Theorem \ref{p^n is solvable}, each $I_j$ is a solvable skew brace. Hence $B$ is a soluble skew brace.
\end{proof}

Two further implications follow when we specialize to the classes considered above.

\begin{pro}\label{soluble group iff soluble skew brace}
    Let $(B, +, \circ)$ be a finite skew brace in the class $\mathcal{T}$. Then $B$ is a soluble (nilpotent)  skew brace if and only if  $(B,+)$ and $(B,\circ)$ are soluble (nilpotent)  groups.
\end{pro}
\begin{proof}
   The solvability assertion follows from \cite[Theorem 3.11, Theorem 3.13]{ST2023} and \cite[Theorem 4.20]{T2023} using induction on the order of the skew brace. If $B$ is nilpotent, then both groups $(B,+)$ and $(B,\circ)$ are nilpotent (even in the general case). So assume that $(B,+)$ and $(B,\circ)$ are nilpotent groups. Then, by \cite[Corollary 4.3.]{CSV19}, $B$ is a direct product of its Sylow ideals, say, $I_1\times I_2\times \cdots \times I_k $, where $|I_j| = p_j^{n_j}$ for distinct primes  $p_j$.
   Note that all finite skew braces of prime-power order are left nilpotent (\cite[Proposition 4.4.]{CSV19}). That two-sided skew braces of prime power order are right nilpotent follows from  \cite[Lemma 4.2]{MY26}. The same assertion for symmetric skew braces of prime-power order follows from \cite[Proposition 4.4]{CSV19} and \cite[Corollary 3.8]{T2023}. Finally, the assertion for $\lambda$-homomorphic skew braces follows from \cite[Theorem 3.13.]{ST2023}.  As the direct product of nilpotent skew braces is also nilpotent, the proof is complete.
\end{proof}

\begin{cor} 
Let $B$ be a finite skew brace such that $\Pb_{\Syl}(B)>2/3$.  Then
$B$ is a direct product of its Sylow ideals, and therefore  solvable. Moreover, if $B$ belongs to the class $\mathcal{T}$, then $B$ is a nilpotent skew brace. 
\end{cor}
\begin{proof}
By Lemma~\ref{lem:sylow-brace-prob-to-groups}, both $(B,+)$ and
$(B,\circ)$ satisfy the group-theoretic Sylow commuting probability
hypothesis with constant $>2/3$. Hence, by \cite{DLMS25}, both groups $(B,+)$ and $(B,\circ)$ are nilpotent. By \cite[Corollary~4.3]{CSV19}, $B$ is a direct product of its Sylow ideals; in particular, it is solvable by Corollary~\ref{pq solvable}(ii). The assertion in the special cases follows from 
Proposition \ref{soluble group iff soluble skew brace}, and the proof is complete.
\end{proof}
\begin{cor} 
Let $B$ be a finite skew brace such that $\Pb_{\Syl}(B)>2/5$.  Then
both groups $(B,+)$ and $(B,\circ)$ are soluble. 
Moreover, if $B$ belongs to the class $\mathcal T$, then $B$ is a soluble skew brace. If, in addition, $B$ is two-sided, then the additive derived series of $B$ consists of ideals of $B$.

\end{cor}

\begin{proof}
By Lemma~\ref{lem:sylow-brace-prob-to-groups}, both underlying groups
satisfy the group-theoretic Sylow commuting probability hypothesis with
constant $>2/5$. Hence, both groups are solvable by \cite{DLMS25}. 
Now assume that $B$ belongs to the class $\mathcal T$. Then by Proposition~\ref{soluble group iff soluble skew brace}, $B$ is a soluble skew brace. If, in addition, $B$ is two-sided, every term of the additive derived series is characteristic in $(B,+)$, and therefore is an ideal of $B$ by \cite[Lemma 4.1]{Nas19}.
\end{proof}

\begin{example} 
The conclusion concerning the additive derived series does not extend to symmetric skew braces. Indeed, there exists a finite symmetric skew brace $B$ such that $\Pb_{\Syl}(B)>2/5$, while $\gamma_2(B,+)$ is not an ideal of $B$.
\end{example}

\begin{proof}
Let $V=\mathbb F_2^2$, $H=\operatorname{GL}_2(\mathbb F_2)\cong S_3$ and $B=V\times H$. Define
$$
(u,h)+(v,g)=(u+v,hg),\,\,\,
\lambda_{(u,h)}(v,g)=(^{h^{-1}}v,h^{-1}gh),\,\,\,
(u,h)\circ(v,g)=(u+ \;^{h^{-1}}v,gh),
$$
where $^{h^{-1}}v$ is the action of $h$ on $v$.
Then $\lambda_{a+b}=\lambda_b\lambda_a$ and $\lambda_{a\circ b}=\lambda_a\lambda_b$ for all $a,b\in B$.  Hence $B$ is a symmetric skew brace. Moreover,
$$
(B,+)\cong C_2^2\times S_3,\qquad
(B,\circ)\cong V\rtimes H\cong \operatorname{AGL}_2(\mathbb F_2)\cong S_4,
$$
where an isomorphism $(B,\circ)\to V\rtimes H$ is given by $(u,h)\mapsto(u,h^{-1})$. Thus both underlying groups are soluble.

Let
$$
r=\begin{pmatrix}0&1\\1&1\end{pmatrix},\qquad
s=\begin{pmatrix}0&1\\1&0\end{pmatrix},\qquad
v=\binom10.
$$
Since $\gamma_2(H)=\langle r\rangle$, we have $\gamma_2(B,+) = \{0\}\times\langle r\rangle$. Put $x :=(v,I)$ and $d:=({\bf 0},r)$. Since several operations are involved in this construction, to avoid confusion, inverse of the element $x$ in $(B, \circ)$ is denoted by $x^{-1_{\circ}}$. Note that  $x = x^{-1_{\circ}}$ and  $r^{-1}v=\binom11$. Thus, we obtain
$$
x\circ d\circ x^{-1_\circ}
 =\bigl(v+r^{-1}v,r\bigr)
 =\left(\binom01,r\right)\notin \{0\}\times\langle r\rangle= \gamma_2(B,+).
$$
Hence $\gamma_2(B,+)$ is not normal in $(B,\circ)$, and therefore it is not an ideal of $B$.

It remains to check the Sylow-local probability. Set $P:=V\rtimes\langle s\rangle$ and $Q:=\{0\}\times\langle r\rangle$, so that $|P|=8$ and $|Q|=3$. For $x=(u,h)\in P$ and $y=(0,r^i)\in Q$, the elements $x$ and $y$ commute if and only if $hr^i=r^ih$ and $r^{-i}u=u$. For $i=0$ all $8$ elements of $P$ commute with $y$. For $i=1,2$, we have $\C_H(r^i)=\langle r\rangle$, $\langle r\rangle\cap\langle s\rangle=\{I\}$, and $r^i$ has no non-zero fixed point on $V$. Hence only $(0,I)\in P$ commutes with $(0,r^i)$. Consequently
$$
\Pb(P,Q)=\frac{8+1+1}{8\cdot3}=\frac5{12}>\frac25.
$$
Therefore, by the definition of $\Pb_{\Syl}(B)$,
$$
\Pb_{\Syl}(B)\geqslant\Pb(P,Q)=\frac5{12}>\frac25,
$$
which proves the claim.
\end{proof}

\begin{cor}
      If $B$ belongs to the class $\mathcal{T}$ and $|B|$ is of the form $p^sq^t$, $pqr$, or is an odd integer, then $B$ is a soluble skew brace, where $p,q,r$ are primes and $s, t$ are positive integers.  
\end{cor}
\begin{proof}
    In all these cases, both groups $(B,+)$ and $(B,\circ)$ are solvable. The proof now follows from Proposition \ref{soluble group iff soluble skew brace}.
\end{proof}

We remark that skew braces $B$ of order $p^2q$ in general  need not be solvable, while both groups $(B,+)$ and $(B,\circ)$ are solvable. For example, consider the skew brace $B$ of order $12$ given by GAP \cite{GAP} with ID $[12,22]$.  Using \cite{VK22}, one can easily see that $B$ is a  simple skew brace with $(B,+)\cong \mathbb{A}_4$ and $(B,\circ) \cong C_3\rtimes C_4$. Note that all skew braces of order $< 12$ are solvable.

It was proved in \cite[Theorem 4.5.]{MY26} that if $B$ is a finite skew brace such that $\Pb(B) > \frac{65}{128}$, then $B$ is nilpotent. We conclude with some improvements of this bound for certain classes of skew braces.

\begin{cor}\label{1\2 and 1\12 solvable}
The following statements hold: (1) If $(B,+,\circ)$ is a finite skew brace such that $\Pb(B)>\frac12$, then $B$ is solvable. Moreover, if $B$ belongs to the class $\mathcal{T}$, then $B$ is a nilpotent skew brace; (2) If $B$ belongs to the class $\mathcal{T}$ such that $\Pb(B) > \frac{1}{12}$, then $B$ is solvable.
\end{cor}
\begin{proof}
   We see by Lemma \ref{lem:sylow-brace-prob-to-groups} that $\Pr (B,+)>\epsilon $ and $\Pr (B,\circ)>\epsilon$ whenever $\Pb(B) > \epsilon$. If $\Pb(B)>\frac12$, then by \cite[Theorem 5.1.]{PL95}, $(B,+)$ and $(B,\circ)$ are nilpotent groups. So by Corollary \ref{pq solvable} (ii) $B$ is  solvable. If $B$ belongs to the class $\mathcal T$ and $\Pb(B)>\frac12$, then $B$ is a nilpotent skew brace by Proposition~\ref{soluble group iff soluble skew brace}, and assertion (1) follows. Assertion (2) follows from Dixon's $\frac{1}{12}$ theorem \cite{D73}  and Proposition \ref{soluble group iff soluble skew brace}.
\end{proof}

 We remark that the bounds in the preceding result are sharp.  Let $(B,+,\circ)= (S_3,+,+)$. Then, using \cite{VK22}, we see that  $\Pb(B) = \frac{1}{2}$, but $B$ is not a nilpotent skew brace.  Let $(B,+,\circ)= (A_5,+,+)$. Then $\Pb(B) = \frac{1}{12}$ and $B$ is not a solvable skew brace.


\section{An Application to Solutions and Their Structure Skew Braces}

In this section we introduce the commuting probability of a finite set-theoretic solution of the Yang--Baxter equation and present some applications  to its structure skew brace. Let $(X,r)$ be a finite non-degenerate solution of the Yang--Baxter equation, and write
$r(x,y)=(\sigma_x(y),\tau_y(x))$.  We define the
\emph{commuting probability} of $(X,r)$ by
$$\operatorname{p}(X,r) :=\frac{1}{|X|^2}
\big|\{(x,y)\in X^2:r(x,y)=(y,x)\text{ and }r(y,x)=(x,y)\}\big|.
$$
Thus $\operatorname{p}(X,r)$ is defined only in terms of the solution. Let $\Cy_{(X,r)}(x) :=\{y\in X\mid r(x,y)=(y,x)\text{ and }r(y,x)=(x,y)\}$. Then $\operatorname{p}(X,r)$ can be written as 
$$\operatorname{p}(X,r) =\frac{1}{|X|^2}\sum_{x\in X}|\Cy_{(X,r)}(x)|.
$$

Note that $\operatorname{p}(X,r)=1$ if and only if $r(x,y)=(y,x)$ for all $x,y\in X$ if and only if $(X, r)$ is a trivial solution. 
For an involutive solution $(X,r)$, we have
$\tau_y(x)=\sigma^{-1}_{\sigma_x(y)}(x)$. Since $r^2=\Id_{X\times X}$, the condition
$r(x,y)=(y,x)$ already implies $r(y,x)=(x,y)$. Therefore
$y\in\Cy_{(X,r)}(x)$ if and only if
$\sigma_x(y)=y$ and $\tau_y(x)=x$. Under the first equality,
$\tau_y(x)=\sigma_y^{-1}(x)$, so the second equality is equivalent to
$\sigma_y(x)=x$. Consequently,
$$
\operatorname{p}(X,r)=\frac{1}{|X|^2}\sum_{x\in X}
|\{y\in X\mid \sigma_x(y)=y\text{ and }\sigma_y(x)=x\}|.
$$


\begin{lem}\label{lem:ybe-rn2-local-global}
Let $(X,r)$ be a finite non-degenerate solution of size $n$ such that its structure skew brace $G=G(X,r)$ lies in the class $\mathcal{T}$. Set
$q_n := (n!)^2$, $k_n := q_n-1$ and $c_n := (2n/(2n+1))^2$. If
$\operatorname{p}(X,r)\geqslant\epsilon>0$, then, for the quotient skew brace $Q := G/\Soc(G)$, we get
$$
\Pb(Q)\geqslant\eta_n(\epsilon):=
\frac{(c_n\epsilon)^{k_n^2}}{(n!)^4}.
$$
\end{lem}

\begin{proof}
 If $n=1$, then $Q=1$ and the conclusion is immediate. Hence assume that $n\geqslant2$.
 Since $G$ lies in the class $\mathcal{T}$, so does $Q$. Proposition~\ref{TS-S-L} now implies that $\Cb_Q(a)$ is a sub-skew brace of $Q$ for every $a\in Q$.
Let $K$ be the kernel of the natural permutation homomorphism
$$
G\longrightarrow \operatorname{Sym}(X)\times\operatorname{Sym}(X).
$$
By the proof of \cite[Theorem~3.11]{Bac18}, $K\leqslant\Soc(G)$. Hence $Q=G/\Soc(G)$ is a quotient of $G/K$, while $G/K$ is a subgroup of
$\operatorname{Sym}(X)\times\operatorname{Sym}(X)$. In particular, $Q$ is finite and is generated by
the images of the elements of $X$. Consequently
$$
|Q|\leqslant(n!)^2=q_n,
$$
and  $(Q, \circ)$ is generated by the elements
$\bar x$, $x\in X$.

Let
$$
E=\{(x,y)\in X^2:r(x,y)=(y,x)\text{ and }r(y,x)=(x,y)\}.
$$
The canonical map $X\to G$ is a morphism from $(X,r)$ to the solution
associated with $G$, and the quotient map $G\to Q$ is also a morphism of
solutions. Thus, if $(x,y)\in E$, then
$r_Q(\bar x,\bar y)=(\bar y,\bar x)$ and
$r_Q(\bar y,\bar x)=(\bar x,\bar y)$. For the solution associated with a
skew brace, these two identities are equivalent to
$\bar y\in\Cb_Q(\bar x)$.

Since $\Cb_Q(\bar x)$ is a sub-skew brace, it contains $\bar y^{-1}$.
Moreover the relation $b\in\Cb_Q(a)$ is symmetric in $a$ and $b$. Hence
$\bar y\in\Cb_Q(\bar x)$ also gives
$\bar y^{\delta}\in\Cb_Q(\bar x^{\varepsilon})$ for all
$\varepsilon,\delta\in\{1,-1\}$.

Set
$$
\mathcal A=\{1\}\cup\{(x,+),(x,-):x\in X\}
$$
and define $\pi(1)=1$, $\pi(x,+)=\bar x$ and
$\pi(x,-)=\bar x^{-1}$. For $a,b\in\mathcal A$, write $a\sim b$ when
$\pi(b)\in\Cb_Q(\pi(a))$, and let $D$ be the number of pairs
$(a,b)\in\mathcal A^2$ such that $a\sim b$. Each pair $(x,y)\in E$ gives
the four pairs $((x,\varepsilon),(y,\delta))$, with
$\varepsilon,\delta\in\{+,-\}$. Therefore, writing $m=|\mathcal A|=2n+1$,
$$
D\geqslant4|E|=4n^2\operatorname{p}(X,r)\geqslant4n^2\epsilon
=c_n\epsilon m^2.
$$

We now count pairs of words whose letters are pairwise related. Put $k=k_n$.
For $(a_1,\ldots,a_k)\in\mathcal A^k$, let
$N(a_1,\ldots,a_k)$ be the number of $b\in\mathcal A$ such that
$a_i\sim b$ for every $i$. If $\deg(b)$ denotes the number of $a$ satisfying
$a\sim b$, then
$$
\sum_{(a_1,\ldots,a_k)\in\mathcal A^k}N(a_1,\ldots,a_k)
=\sum_{b\in\mathcal A}\deg(b)^k.
$$
Since $\sum_b\deg(b)=D$, the inequality
$(u_1+\cdots+u_m)^k\leqslant m^{k-1}(u_1^k+\cdots+u_m^k)$ for non-negative
$u_1,\ldots,u_m$ gives
$$
\sum_{b\in\mathcal A}\deg(b)^k\geqslant\frac{D^k}{m^{k-1}}.
$$
Applying the same inequality to the $m^k$ numbers
$N(a_1,\ldots,a_k)$, we get
$$
\sum_{(a_1,\ldots,a_k)\in\mathcal A^k}N(a_1,\ldots,a_k)^k
\geqslant
\frac{D^{k^2}}{m^{2k(k-1)}}
=
\left(\frac{D}{m^2}\right)^{k^2}m^{2k}
\geqslant(c_n\epsilon)^{k^2}m^{2k}.
$$
The left hand side is precisely the number of pairs
$(a_1,\ldots,a_k),(b_1,\ldots,b_k)\in\mathcal A^k$ such that
$a_i\sim b_j$ for all $i,j$. Hence there are at least
$(c_n\epsilon)^{k^2}m^{2k}$ such pairs of words.

Every element of $(Q,\circ)$ can be represented by a word of length at most
$|Q|-1\leqslant k$ in the generators $\bar x^{\pm1}$: for example, take a
shortest word and observe that its successive partial products are distinct.
If the word has length smaller than $k$, insert copies of the identity.
Thus the map
$$
\Phi:\mathcal A^k\longrightarrow Q,\qquad
\Phi(a_1,\ldots,a_k)=\pi(a_1)\circ\cdots\circ\pi(a_k),
$$
is surjective.

Assume that $a_i\sim b_j$ for every $i,j$. For each $i$, all the elements
$\pi(b_j)$ belong to the sub-skew brace $\Cb_Q(\pi(a_i))$, so their product
$\Phi(b_1,\ldots,b_k)$ also belongs to $\Cb_Q(\pi(a_i))$. By symmetry,
each $\pi(a_i)$ belongs to
$\Cb_Q(\Phi(b_1,\ldots,b_k))$. Since this centralizer is again a sub-skew
brace, it contains the product $\Phi(a_1,\ldots,a_k)$. Therefore the elements $\Phi(a_1, a_2, \ldots a_k)$ and $\Phi(b_1, b_2, \ldots b_k)$  brace-commute in $Q$.

A fixed pair of elements of $Q$ has at most $m^{2k}$ preimages under
$\Phi\times\Phi$. Hence the number of brace-commuting pairs in $Q^2$ is at
least $(c_n\epsilon)^{k^2}$. Using $|Q|\leqslant(n!)^2$ and $k=k_n$, we obtain
$$
\Pb(Q)\geqslant\frac{(c_n\epsilon)^{k^2}}{|Q|^2}
\geqslant\frac{(c_n\epsilon)^{k_n^2}}{(n!)^4}.
$$
The proof is now complete.
\end{proof}

We are now ready to prove the main result of this section (Theorem \ref{main-3}).

 \begin{thm} 
Let $(X,r)$ be a finite non-degenerate solution of size $n$ such that its structure skew brace $G:=G(X,r)$ lies in the class $\mathcal{T}$  and 
$\operatorname{p}(X,r)\geqslant\epsilon>0$. Then
$|G:\Soc(G)|$ is bounded in terms of $\epsilon$ and $n$ only.
More precisely, set $\eta :=\eta_n(\epsilon)$ and
$L_n :=\operatorname{lcm}(1,\ldots,n)$. Let $M(\eta)$ and $N(\eta)$ be bounds
in Theorem~\ref{thm:bfc-ideals-class-T}, so that the ideals $Q_1$ and $Q_2$ given
by that theorem satisfy $|Q:Q_1|\leqslant M(\eta)$ and
$|Q_2|\leqslant N(\eta)$. Then
$$
|G :\Soc(G)|
\leqslant M(\eta)N(\eta)L_n^{2M(\eta)n}.
$$ 
\end{thm}
\begin{proof}
Set $Q: =G/\Soc(G)$. By
Lemma~\ref{lem:ybe-rn2-local-global}, $\Pb(Q)\geqslant\eta$. The skew brace $Q$
lies in the class $\mathcal{T}$, and  so
Theorem~\ref{thm:bfc-ideals-class-T} gives ideals
$Q_2\leqslant Q_1\trianglelefteq Q$ such that
$$
|Q_2|\leqslant N(\eta),\qquad |Q:Q_1|\leqslant M(\eta),
\qquad Q_1/Q_2\text{ is a trivial brace}.
$$
In particular, $Q_1/Q_2$ is a trivial skew brace and its additive and
multiplicative groups are abelian and coincide.

It remains to bound $|Q_1/Q_2|$. The group $(Q,\circ)$ is generated by the
$n$ elements $\bar x$, $x\in X$. Since, as observed in the preceding proof, $(Q,\circ)$ is a quotient of a subgroup of
$\operatorname{Sym}(X)\times\operatorname{Sym}(X)$, every element of
$(Q,\circ)$ has order dividing
$$
L_n=\operatorname{lcm}(1,\ldots,n).
$$

Put $m=|Q:Q_1|$ and choose representatives $T$ for the right cosets of
$(Q_1,\circ)$ in $(Q,\circ)$, with $1\in T$. Let
$S=\{\bar x^{\pm1}:x\in X\}$. For every $t\in T$ and $s\in S$, write
$$
t\circ s=h_{t,s}\circ t',
$$
where $h_{t,s}\in Q_1$ and $t'\in T$. We claim that the elements $h_{t,s}$
generate $(Q_1,\circ)$. Indeed, let
$u=s_1\circ\cdots\circ s_\ell\in Q_1$, where $s_i\in S$. Starting with
$t_0=1$, write recursively
$t_{i-1}\circ s_i=h_i\circ t_i$, with $h_i\in Q_1$ and $t_i\in T$.
Multiplying these equalities gives
$$
u=h_1\circ\cdots\circ h_\ell\circ t_\ell.
$$
Since $u\in Q_1$, the last representative is $t_\ell=1$. Hence
$u=h_1\circ\cdots\circ h_\ell$, proving the claim.

There are at most $2mn$ elements $h_{t,s}$. Hence the multiplicative group of
$Q_1/Q_2$ is generated by at most
$2mn\leqslant2M(\eta)n$ elements. It is abelian because $Q_1/Q_2$ is a trivial
brace, and its exponent divides $L_n$. Therefore
$$
|Q_1/Q_2|\leqslant L_n^{2M(\eta)n}.
$$
Finally,
$$
|G:\Soc(G)|=|Q|=|Q:Q_1|\,|Q_2|\,|Q_1/Q_2|
\leqslant M(\eta)N(\eta)L_n^{2M(\eta)n}.
$$
This completes the proof.
\end{proof}

By \cite[Remark~3.20]{JLPTVA26}, the structure skew brace associated with a solution $(X,r)$ is right nilpotent of class at most $2$ if and only if
$\sigma_{\sigma_x(y)}=\sigma_y$ for every $x,y\in X$, where $r(x, y) = (\sigma_x(y), \tau_y(x))$.
By \cite[Corollary~2.10]{JLPTVA26}, a right nilpotent skew brace of class at
most $2$ is $\lambda$-homomorphic. Similar conditions ensuring that the associated skew brace is two-sided or symmetric can be obtained.

\begin{example} 
Even in the injective case, the hypothesis
$\operatorname{p}(X,r)\geqslant\epsilon$ does not bound $|X|$ in terms of
$\epsilon$ alone. In fact, there is no function $f:(0,1]\to\mathbb N$ such
that every finite injective right nilpotent solution of class at most $2$
with $\operatorname{p}(X,r)\geqslant\epsilon$ satisfies $|X|\leqslant f(\epsilon)$.
\end{example}
\begin{proof}
Let $X$ be any finite set and consider the trivial solution
$r(x,y)=(y,x)$. The defining relations of its structure group are
$xy=yx$ for $x,y\in X$, so $G(X,r)$ is the free abelian group with basis
$X$. Hence the canonical map $X\to G(X,r)$ is injective. Moreover
$\sigma_x=\operatorname{id}_X$ for every $x\in X$, so the solution is right
nilpotent of class $1$, and every pair belongs to the set defining
$\operatorname{p}(X,r)$. Thus $\operatorname{p}(X,r)=1$. Since $|X|$ is
arbitrary, no bound depending only on $\epsilon$ can hold.
\end{proof}

Note that the commuting probability $\operatorname{p}(B,r_B)$ of the solution $(B,r_B)$ associated with the finite skew brace $B:=(B,+,\circ)$ is equal to the commuting probability $\Pb(B)$ of $B$. Thus $\operatorname{p}(B,r_B)=1$, $\operatorname{p}(B,r_B)=\frac34$, or $\operatorname{p}(B,r_B)\leqslant\frac58$. We now recall the following result connecting nilpotency of a finite skew brace with the multipermutation level of its associated solution.
\begin{thm} 
    \cite[Theorem~2.20]{CSV19} Let $B$ be a skew left brace. Then the associated solution $(B,r_B)$ has finite multipermutation level if and only if $B$ is right nilpotent and $(B,+)$ is a nilpotent group.
\end{thm}
Note that every nilpotent skew brace $B$ is right nilpotent with $(B,+)$ a nilpotent group. Hence $(B,r_B)$ is a multipermutation solution of finite level whenever $\operatorname{p}(B,r_B)>\frac{65}{128}$. Moreover, if $B$ belongs to the class $\mathcal T$, then the same conclusion follows from $\operatorname{p}(B,r_B)>\frac12$. 

We now make some remarks on decomposability of the solution $(B, r_B)$ in terms of the commuting probability.

\begin{pro}\cite[Proposition 2.5]{JKAV19}\label{strong ideal and SYBE}
    Let $B$ be a skew brace. If there exists a proper strong left ideal $I$ of $B$, then $(B,r_B)$ is decomposable.
\end{pro} 

\begin{rem}
    Let $(B,+,\circ)$ be a finite skew brace such that $\Pb(B)>\frac{1}{12}$. Then either $B$ is a brace or $(B,r_B)$ is decomposable. Moreover, if $B$ belongs to the class $\mathcal{T}$, then $B$ is a trivial brace or $(B,r_B)$ is decomposable.
\end{rem}
\begin{proof}
     By Dixon's $\frac{1}{12}$ theorem \cite{D73}, $(B,+) $ is a solvable group. Every characteristic subgroup of $(B,+)$ is a strong left ideal of $B$. Hence the first assertion follows from Proposition \ref{strong ideal and SYBE}.
   The final assertion  follows from Corollary \ref{1\2 and 1\12 solvable} and Proposition \ref{strong ideal and SYBE}. 
\end{proof}


\end{document}